\documentclass[12pt]{amsart}
\usepackage{amssymb}
\usepackage[shortlabels]{enumitem}
\usepackage[all]{xy}
\usepackage{xcolor}
\usepackage[margin=1in]{geometry} 
\usepackage{hyperref}
\usepackage{mathrsfs}
\usepackage{eucal}
\usepackage{contour}
\usepackage{dsfont}
\usepackage{shuffle}
\usepackage[normalem]{ulem}
\hypersetup{bookmarksdepth=2}
\hypersetup{colorlinks=true}
\hypersetup{linkcolor=blue}
\hypersetup{citecolor=blue}
\hypersetup{urlcolor=blue}
\usepackage{tikz}

\numberwithin{equation}{section}
\newtheorem{theorem}[equation]{Theorem}
\newtheorem{proposition}[equation]{Proposition}
\newtheorem{lemma}[equation]{Lemma}
\newtheorem{corollary}[equation]{Corollary}
\newtheorem{maintheorem}{Theorem}

\theoremstyle{definition}
\newtheorem{remark}[equation]{Remark}

\newtheorem{example}[equation]{Example}
\newtheorem{definition}[equation]{Definition}

\newcommand{\cC}{\mathcal{C}}
\newcommand{\fC}{\mathfrak{C}}

\newcommand{\sE}{\mathscr{E}}
\newcommand{\bF}{\mathbf{F}}

\newcommand{\fF}{\mathfrak{F}}

\newcommand{\bN}{\mathbf{N}}

\newcommand{\fP}{\mathfrak{P}}

\newcommand{\bQ}{\mathbf{Q}}

\newcommand{\fQ}{\mathfrak{Q}}

\newcommand{\bR}{\mathbf{R}}

\newcommand{\bS}{\mathbf{S}}
\newcommand{\cS}{\mathcal{S}}
\newcommand{\fS}{\mathfrak{S}}

\newcommand{\fT}{\mathfrak{T}}

\newcommand{\fU}{\mathfrak{U}}

\newcommand{\bV}{\mathbf{V}}

\newcommand{\fX}{\mathfrak{X}}

\newcommand{\sX}{\mathscr{X}}

\newcommand{\bZ}{\mathbf{Z}}

\newcommand{\tto}{\twoheadrightarrow}
\newcommand{\GG}{\mathbb{G}}

\newcommand{\arxiv}[1]{\href{http://arxiv.org/abs/#1}{{\tiny\tt arXiv:#1}}}

\newcommand{\DOI}[1]{\href{http://doi.org/#1}{\color{purple}{\tiny\tt DOI:#1}}}
\newcommand{\defn}[1]{\emph{#1}}

\let\ul\underline

\newcommand{\op}{\mathrm{op}}
\renewcommand{\phi}{\varphi}
\newcommand{\Set}{\mathbf{Set}}
\newcommand{\FinSet}{\mathbf{FinSet}}
\newcommand{\OI}{\mathbf{OI}}
\DeclareMathOperator{\tr}{tr}

\DeclareMathOperator{\im}{im}

\DeclareMathOperator{\Et}{Et}

\DeclareMathOperator{\Rep}{Rep}

\DeclareMathOperator{\Sub}{Sub}
\DeclareMathOperator{\LEx}{LEx}
\DeclareMathOperator{\End}{End}

\DeclareMathOperator{\Aut}{Aut}
\DeclareMathOperator{\OrdEt}{OrdEt}
\DeclareMathOperator{\Ord}{Ord}
\DeclareMathOperator{\Hom}{Hom}
\DeclareMathOperator{\Isom}{Isom}

\DeclareMathOperator{\Fun}{Fun}

\DeclareMathOperator{\Del}{Del}
\newcommand{\id}{\mathrm{id}}
\renewcommand{\Vec}{\mathrm{Vec}}
\newcommand{\Ver}{\mathrm{Ver}}
\newcommand{\GL}{\mathbf{GL}}
\newcommand{\bbone}{\mathds{1}}
\newcommand{\bzero}{\mathbf{0}}
\newcommand{\bone}{\mathbf{1}}

\newcommand{\Mod}{\mathrm{Mod}}
\newcommand{\wa}{{\bullet}}
\newcommand{\wb}{{\circ}}
\newcommand{\isom}{\mathrm{isom}}

\contourlength{1pt}

\contourlength{0.8pt}
\newcommand{\myuline}[1]{%
  \uline{\phantom{#1}}%
  \llap{\contour{white}{#1}}%
}
\DeclareMathOperator{\uRep}{\text{\myuline{\rm Rep}}}
\DeclareMathOperator{\uPerm}{\ul{Perm}}

\title{Universal properties of Delannoy categories}

\author{Kevin Coulembier}
\address{School of Mathematics and Statistics, University of Sydney, NSW 2006, Australia}
\email{\href{mailto:kevin.coulembier@sydney.edu.au}{kevin.coulembier@sydney.edu.au}}
\urladdr{\url{https://www.maths.usyd.edu.au/u/kevinc/}}
\thanks{KC was supported by ARC grant FT220100125}

\author{Nate Harman}
\address{Department of Mathematics, University of Georgia, Athens, GA, USA}
\email{\href{mailto:nharman@uga.edu}{nharman@uga.edu}}
\urladdr{\url{https://www.nateharman.com/}}
\thanks{NH was supported by NSF grant DMS-2401515}

\author{Andrew Snowden}
\address{Department of Mathematics, University of Michigan, Ann Arbor, MI}
\email{\href{mailto:asnowden@umich.edu}{asnowden@umich.edu}}
\urladdr{\url{http://www-personal.umich.edu/~asnowden/}}
\thanks{AS was supported by NSF grant DMS-2301871.}

\begin{document}

\begin{abstract}
Recently, the second and third authors introduced a new symmetric tensor category $\uPerm(G, \mu)$ associated to an oligomorphic group $G$ with a measure $\mu$. When $G$ is the group of order preserving self-bijections of the real line there are four such measures, and the resulting tensor categories are called the Delannoy categories. The first Delannoy category is semi-simple, and was studied in detail by Harman, Snowden, and Snyder. We give universal properties for all four Delannoy categories in terms of ordered \'etale algebras. As a consequence, we show that the second and third Delannoy categories admit at least two local abelian envelopes, and the fourth admits at least four. We also prove a coarser universal property for $\uPerm(G, \mu)$ for a general oligomorphic group $G$.
\end{abstract}

\maketitle
\tableofcontents

\section{Introduction}

Deligne \cite{Deligne} introduced an important tensor category $\uRep(\fS_t)$ by ``interpolating'' the representation categories of finite symmetric groups. He showed that this category can be characterized by a universal property: giving a tensor functor from $\uRep(\fS_t)$ to a tensor category $\fT$ is equivalent to giving an \'etale algebra in $\fT$ of dimension $t$. Knop \cite{Knop1, Knop2} studied a related category $\uRep(\GL_t(\bF_q))$, and a universal property for it has been found as well \cite{EntovaAizenbudHeidersdorf}. Recently, two of us \cite{repst} constructed a tensor category $\uPerm(G, \mu)$ associated to any oligomorphic group $G$ and measure $\mu$. This construction recovers Deligne's when $G$ is the infinite symmetric group and Knop's when $G$ is the infinite general linear group over $\bF_q$. It is therefore natural to ask if the categories $\uPerm(G, \mu)$ always have universal properties.

In this paper, we establish a rough mapping property in full generality. We hone this when $G=\GG$ is the Delannoy group $\Aut(\bR, <)$ of order preserving self-bijections of the real line to obtain a very precise mapping property for the Delannoy categories, which are some of the most important categories coming from the theory of \cite{repst}. As an application, we show that the three non-abelian Delannoy categories admit multiple local abelian envelopes. This means that some interesting new pre-Tannakian categories must exist. The first and third authors will describe such categories in more detail in future work \cite{fake}.

\subsection{A general mapping property}

Fix a field $k$. In this paper, a \defn{tensor category} is a $k$-linear symmetric monoidal category, and a \defn{tensor functor} is a $k$-linear symmetric monoidal functor; see \S \ref{ss:tencat} for details. Fix an oligomorphic group $G$ with a $k$-valued measure $\mu$ (see \S \ref{s:oligo} for background), and let $\fT$ be an arbitrary tensor category. Our first goal is to give a description of tensor functors $\uPerm(G, \mu) \to \fT$.

Before stating our result, we must recall some basic concepts. In any tensor category $\fT$, one can define the notion of \'etale algebra (\S \ref{ss:etale}). We write $\Et(\fT)$ for the category of \'etale algebras in $\fT$. One should regard $\Et(\fT)$ as an object of a combinatorial nature. For instance, if $\fT$ is the category of representations of a finite group $\Gamma$ then $\Et(\fT)^{\op}$ is the category $\bS(\Gamma)$ of finite $\Gamma$-sets. The situation is somewhat similar for the categories $\uPerm(G, \mu)$. Let $\bS(G)$ denote the category of finitary smooth $G$-sets (\S \ref{ss:oligo}). For any object $X$ of $\bS(G)$, there is an associated object $\cC(X)$ of $\uPerm(G, \mu)$, which is naturally an \'etale algebra. This construction defines a fully faithful functor $\bS(G) \to \Et(\uPerm(G, \mu))^{\op}$ that is often (though not always) an equivalence. In any case, one should consider the general character of $\Et(\uPerm(G, \mu))^{\op}$ as similar to that of $\bS(G)$.

Suppose now that we have a tensor functor $\Phi \colon \uPerm(G, \mu) \to \fT$. Since $\Phi$ maps \'etale algebras to \'etale algebras, it follows that there is an induced functor $\Psi \colon \bS(G) \to \Et(\fT)^{\op}$. This functor is additive (commutes with finite co-products), left-exact (commutes with finite limits), and compatible with $\mu$ (Definition~\ref{defn:compatible}). The following is our first main result:

\begin{maintheorem} \label{mainthm1}
Giving a tensor functor $\Phi \colon \uPerm(G, \mu) \to \fT$ is equivalent to giving a functor $\Psi \colon \bS(G) \to \Et(\fT)^{\op}$ that is additive, left-exact, and compatible with $\mu$.
\end{maintheorem}

In the body of the paper, we give a more precise result that also explains how isomorphisms of $\Phi$'s correspond to isomorphisms of $\Psi$'s. To prove Theorem~\ref{mainthm1}, we build on a more primitive mapping property for $\uPerm(G, \mu)$ given in \cite{repst}.

\subsection{Finer mapping properties} \label{ss:finer}

Theorem~\ref{mainthm1} is useful since it converts the problem of describing tensor functors from $\uPerm(G, \mu)$ into the more combinatorial problem of describing functors from $\bS(G)$. However, it says nothing about the latter problem. We therefore view Theorem~\ref{mainthm1} as only a first step towards providing a useful universal description of $\uPerm(G, \mu)$. For a given $G$, there is a natural two-step plan to follow to complete the universal description:
\begin{enumerate}
\item Give a universal property for $\bS(G)$, that is, give a characterization of additive left-exact functors $\Psi \colon \bS(G) \to \cS$, where $\cS$ belongs to some class $\sX$ of categories. Of course, we want $\sX$ to include all categories of the form $\Et(\fT)^{\op}$. A convenient choice for $\sX$, which we adopt, is the class of \defn{lextensive categories} (\S \ref{ss:lextensive}). We view this problem as purely combinatorial.
\item In case $\cS=\Et(\fT)^{\op}$, give a characterization of which $\Psi$'s are compatible with a given measure $\mu$. We introduce the notion of $\Theta$-generators in \S \ref{ss:theta-gen} to aid in the solution of this problem.
\end{enumerate}
We carry out this plan in the case of Delannoy categories.

\subsection{Delannoy categories}

Let $\GG$ be the oligomorphic group $\Aut(\bR, <)$. This group carries exactly four $k$-valued measures, which we denote by $\mu_i$ for $1 \le i \le 4$. We put
\begin{displaymath}
\fC_i = \uPerm(\GG, \mu_i)^{\rm kar},
\end{displaymath}
which we refer to as the $i$th \defn{Delannoy category}. The first (or simply `the') Delannoy category $\fC_1$ is semi-simple pre-Tannakian. It was studied in great detail in \cite{line} and shown to have a number of remarkable properties: for example, its simple objects all have categorical dimension $\pm 1$, and the Adams operations on its Grothendieck group are all trivial. The other Delannoy categories have remained somewhat mysterious, but we hope to shed some light on them in this paper and in the forthcoming work \cite{fake}.

The primary purpose of this paper is to give precise universal properties for the $\fC_i$'s. To do this, we follow the plan put forth in \S \ref{ss:finer}. As a first step, we give a universal property for $\bS(\GG)$. This states that additive left-exact functors $\bS(\GG) \to \cS$, with $\cS$ lextensive, correspond to totally ordered objects of $\cS$. (We develop the theory of ordered objects in lextensive categories in \S \ref{ss:ordered}.)

Let $\fT$ be a Karoubian tensor category. We define an \defn{ordered \'etale algebra} in $\fT$ to be a totally ordered object in the lextensive category $\Et(\fT)^{\op}$; see \S \ref{ss:ordet} for an explicit description of the concept. The universal property for $\bS(\GG)$, in this case, shows that giving an additive left-exact functor $\Psi \colon \bS(\GG) \to \Et(\fT)^{\op}$ amounts to giving an ordered \'etale algebra $A$ in $\fT$.

Let $A$ and $\Psi$ be as above. We say that $A$ is a \defn{Delannic algebra} of type $i$ if it satisfies three numerical conditions related to the measure $\mu_i$. For example, a type~1 Delannic algebra must have categorical dimension $-1$; this is one of the three numerical conditions. See \S \ref{ss:delannic} for the complete definition. We show, using the tool of $\Theta$-generators, that $\Psi$ is compatible with the measure $\mu_i$ if and only if $A$ is Delannic of type $i$. This completes the second step of the plan in \S \ref{ss:finer}.

Putting all of the above work together, we reach our second main result:

\begin{maintheorem} \label{mainthm2}
Tensor functors $\fC_i \to \fT$ correspond to type $i$ Delannic algebras in $\fT$.
\end{maintheorem}

The basic objects of the Delannoy category $\fC_i$ are the Schwartz spaces $\cC_i(\bR^{(n)})$. Essentially by definition, $\cC_i(\bR)$ is a type $i$ Delannic algebra. A slightly more precise phrasing of Theorem~\ref{mainthm2} is: if $A$ is a type $i$ Delannic algebra in $\fT$ then there is a unique (up to isomorphism) tensor functor $\fC_i \to \fT$ that maps $\cC_i(\bR)$ to $A$. In other words, $\cC_i(\bR)$ is the universal type $i$ Delannic algebra. See Theorem~\ref{thm:delmap} for a more precise statement still.

\subsection{Examples}

Theorem~\ref{mainthm2} allows us to construct many tensor functors between Delannoy categories. We show that $\cC_1(\bR) \oplus \bbone$ carries an order that makes it a type~2 Delannic algebra in the category $\fC_1$; this ordered algebra corresponds to the $\GG$-set $\bR \cup \{\infty\}$ equipped with its natural order. It follows that there is a tensor functor
\begin{displaymath}
\Phi \colon \fC_2 \to \fC_1, \qquad \Phi(\cC_2(\bR)) = \cC_1(\bR) \oplus \bbone,
\end{displaymath}
which is unique up to isomorphism. This particular functor is one of the primary tools used in \cite{fake} to study the structure of $\fC_2$.

Similarly, the ordered set $\{-\infty\} \cup \bR$ leads to a tensor functor $\fC_3\to \fC_1$, and the ordered set $\{-\infty\} \cup \bR \cup \{\infty\}$ leads to a functor $\fC_4 \to \fC_1$. These functors (and the one from the previous paragraph) are all faithful. We thus see that each of the Delannoy categories admits a faithful tensor functor to $\fC_1$. These examples are significant since, prior to them, we did not know if the $\fC_i$ (for $i \ne 1$) admitted any faithful tensor functor to a pre-Tannakian category.

Let $\bR^{(2)}$ be subset of $\bR^2$ consisting of pairs $(x,y)$ with $x<y$, equipped with the lexicographic order, i.e., $(x,y)<(a,b)$ if $x<a$, or $x=a$ and $y<b$. We show that $\cC_1(\bR^{(2)})$ is a type 4 Delannic algebra in $\fC_1$, resulting in another tensor functor $\fC_4\to \fC_1$. Similarly $\cC_1(\bR^{(n)})$ becomes an ordered algebra, which is type~1 if $n$ is odd, and type~4 if $n$ is even. 

In fact, the above examples are special cases of some general constructions. If $A$ and $B$ are ordered \'etale algebras then $A \oplus B$ and $A \otimes B$ carry orders called the \defn{lexicographic sum} and \defn{lexicographic product}; also there is an algebra $A^{(n)}$ that carries a lexicographic order (and some other orders). We show (\S \ref{ss:delop}) that if $A$ and $B$ are Delannic then (under some constraints) these algebras are as well, with predictable type. For example, since $\cC_1(\bR^{(2)})$ has type~4 and $\cC_1(\bR)$ has type~1, the general constructions show that the leixcographic sum $\cC_1(\bR) \oplus \cC_1(\bR^{(2)})$ has type~2, while the lexicographic product $\cC_1(\bR) \otimes \cC_1(\bR^{(2)})$ has type~1. This enables us to produce a vast quantity of functors between the $\fC_i$'s.

\subsection{Abelian envelopes} \label{ss:intro-abenv}

An important restrictive class of tensor categories is formed by the pre-Tannakian categories; these are the ones that most closely resemble representation categories of groups. See \S \ref{ss:tencat} for the definition. Several recent advances in the theory of pre-Tannakian categories are obtained by first constructing a rigid tensor category before ``completing it'' to a pre-Tannakian category. One instance is the construction of the pre-tannakian categories $\Ver_{p^n}$ in \cite{BEO, AbEnv} which are at the heart of the current study of the structure theory of pre-Tannakian categories of moderate growth in characteristic $p>0$, and are obtained by completing subquotient tensor categories of $\Rep(\mathbf{SL}_2)$.

In general, it is relatively easy to construct rigid tensor categories with certain properties, but typically very difficult to do so with pre-Tannakian cateogries. A standard technique is to establish a pre-Tannakian category as an ``abelian envelope'' of a rigid tensor category. In the context of oligmorphic groups, $\uPerm(G,\mu)^{\rm kar}$ is always a rigid tensor category. If $\mu$ is regular and nilpotent endomorphisms have trace zero (for example $G=\GG$ and $\mu=\mu_1$) then $\uPerm(G,\mu)^{\rm kar}$ is itself a semisimple pre-Tannakian category (and its own abelian envelope). If $\mu$ is quasi-regular and nilpotents have trace zero, then $\uPerm(G,\mu)^{\rm kar}$ admits an abelian envelope $\uRep(G,\mu)$. Crucially for the current paper, $\mu_i$ is not quasi-regular for $i>1$, so we cannot rely on the general theory from \cite{repst}.  

Moreover, it is known that abelian envelopes need not always exist. However, in \cite{HomKer}, the first author established that for any rigid tensor category $\fT$ with $\End(\bbone)=k$ there is a family of faithful tensor functors $\{ \fT \to \fU_i \}_{i \in I}$, with each $\fU_i$ pre-Tannakian, such that any faithful tensor functor from $\fT$ to a pre-Tannakian category factors uniquely through a unique $\fU_i$. These $\fU_i$ are called the \defn{local abelian envelopes} of $\fT$.

There are cases when there is no local abelian envelope, meaning $\fT$ does not map faithfully to any pre-Tannakian category. For example, this happens if there are nilpotent endomorphisms in $\fT$ with non-zero trace. The category $\fT$ admits an abelian envelope\footnote{Here we do not require the functor to the abelian envelope to be full. In some references, this condition is imposed.} precisely when it admits a unique local abelian envelope $\fU$, and then $\fU$ is the abelian envelope. There also also cases where there are infinitely many local abelian envelopes. At present there are no known examples having multiple but finitely many local abelian envelopes.

Determing the local abelian envelopes for the Delannoy categories $\fC_i$ is a natural refinement of the study of their universal properties, wherein the target categories are restricted to being pre-Tannakian. Using the tensor functors between various Delannoy categories provided by our main theorem, we are able to shed some light on this problem. We note that, as explained above, this problem is trivial for $\fC_1$.

\begin{maintheorem}
The categories $\fC_2$ and $\fC_3$ admit at least two local abelian envelopes. The category $\fC_4$ admits at least four.
\end{maintheorem}

\subsection{Combinatorics in tensor categories}

One can make sense of essentially any kind of relational structure (graph, tree, order, etc.) in a lextensive category. Thus, in any tensor category, one can speak of \'etale algebras equipped with such a structure. It is an interesting problem to determine what constraints there are on such objects. This is, in a sense, the central theme of the oligomorphic approach to tensor categories.

This paper is concerned with the special case of total orders in tensor categories, i.e., what we called ordered \'etale algebras. We prove a few general results about such algebras. We mention one here:

\begin{maintheorem}
In a pre-Tannakian tensor category over a separably closed field, a simple ordered \'etale algebra must have dimension $\pm 1$ or~0. Moreover, each of the three possibilities occurs.
\end{maintheorem}

The most difficult part of this theorem is exhibiting an algebra of dimension~0, which we accomplish in \S \ref{ss:dim0}. The construction crucially depends on the universal property of the Delannoy category (Theorem~\ref{mainthm2}).

\subsection{Relation to other work}

In \cite[Theorem~4.9]{Kriz}, Kriz gives a universal property of the first Delannoy category $\fC_1$. In forthcoming work \cite{KS}, Khovanov and Snyder give a variant of Kriz's universal property, and also give diagrammatic interpretations of Kriz's original property and their variant.

Kriz's mapping property again uses the algebra $\cC_1(\bR)$ as the basic object, but it does not use the concept of ordered \'etale algebra. In $\fC_1$, the object $\cC_1(\bR)$ decomposes into three simple representations, and one can record various properties of this decomposition; e.g., one of the summands is the tensor unit, and the other two are in natural duality. Kriz's mapping property essentially state that $\cC_1(\bR)$ is universal with respect to having such a decomposition. The variant mapping property in \cite{KS} is similar, but it focuses on one of the individual summands instead of all of $\cC_1(\bR)$.

It is not entirely obvious how these mapping properties align with ours. We plan to address this in forthcoming work \cite{delchar}. It is also not clear if the approach of \cite{Kriz,KS} can be adapted to handle the other Delannoy categories $\fC_i$ with $2 \le i \le 4$.

\subsection{Questions}

We mention a few questions or problems arising from this work.
\begin{itemize}
\item An obvious problem is to understand the local envelopes of the categories $\fC_i$ more thoroughly. This will be solved in \cite{fake} for $\fC_2$ and $\fC_3$. 
\item Another problem is to determine universal properties for categories corresponding to the measures in \cite{colored} and \cite{homoperm}. It would be especially interesting to study local envelopes in these cases.
\item If $A$ is a simple ordered \'etale algebra in a pre-Tannakian category, is $\Gamma(A)=k$? This is trivially true if $k$ is separably closed (this does not even require an order), so the question is really about the case when $k$ is not closed.
\end{itemize}

\subsection{Tensor category terminology} \label{ss:tencat}

We fix a field $k$ for the duration of the paper. A \defn{tensor category} is an additive $k$-linear category equipped with a symmetric monoidal structure that is $k$-bilinear. We write $\bbone$ for the monoial unit in a tensor category $\fT$, and
\begin{displaymath}
\Gamma \colon \fT \to \mathrm{Vec}, \qquad \Gamma(X) = \Hom_{\fT}(\bbone, -)
\end{displaymath}
for the invariants functor. A \defn{tensor functor} is a $k$-linear symmetric monoidal functor. Given tensor categories $\fT$ and $\fT'$, we write $\Fun^{\otimes}(\fT, \fT')$ for the category of tensor functors; the morphisms in this category are monoidal natural transformations. An object of $\fT$ is called \defn{rigid} if it has a dual, and $\fT$ is called \defn{rigid} if every object is. A rigid object $X$ has a categorical dimension $\dim(X)$, which is an element of $\Gamma(\bbone)$; note that, in general, $\Gamma(\bbone)$ is just some $k$-algebra, so the categorical dimension need not be an element of $k$. We say that $\fT$ is \defn{pre-Tannakian} if it is rigid, abelian, all objects have finite length, all $\Hom$ spaces are finite dimensional, and $\End(\bbone)=k$.

\subsection{Notation}

We list some of the important notation here:
\begin{description}[align=right,labelwidth=2.5cm,leftmargin=!]
\item[ $k$ ] the coefficient field
\item[$\bN$ ] the natural numbers, including $0$
\item[ $\bzero$ ] the initial object of a lextensive category
\item[ $\bone$ ] the final object of a lextensive category
\item[ $\bbone$ ] unit object of a tensor category
%\item[ $\fS$ ] the infinite symmetric group
%\item[ $\Omega$ ] the set $\{1,2,3,\ldots\}$ on which $\fS$ acts
\item[ $\GG$ ] the oligomorphic group $\Aut(\bR, <)$
\item[ $\bR^{(n)}$ ] the set of increasing tuples $(x_1, \ldots, x_n)$ in $\bR^n$
\item[ $\fC_i$ ] the $i$th Delannoy category, where $1 \le i \le 4$
\end{description}

\subsection*{Acknowledgments}

We thank Pavel Etingof for helpful discussions.

\section{Oligomorphic groups and their tensor categories} \label{s:oligo}

In this section, we review some essential material about oligomorphic groups and the tensor categories constructed from them. Most of this material is drawn from \cite{repst}. The one exception is the material on $\Theta$-generators in \S \ref{ss:theta-gen}, which is new.

\subsection{Oligomorphic groups} \label{ss:oligo}

An \defn{oligomorphic group} is a permutation group $(G, \Omega)$ such that $G$ has finitely many orbits on $\Omega^n$ for all $n \ge 0$. Fix such a group. For a finite subset $A$ of $\Omega$, let $G(A)$ be the subgroup of $G$ fixing each element of $A$. These subgroups form a neighborhood basis for a topology on $G$. This topology has the following properties \cite[\S 2.2]{repst}: it is Hausdorff; it is non-archimedean, i.e., open subgroups form a neighborhood basis of the identity; and it is Roelcke-precompact, i.e., if $U$ and $V$ are open subgroups then $U \backslash G/V$ is a finite set. A topological group with these three properties is called \defn{pro-oligomorphic}. While most pro-oligomorphic groups of interest are in fact oligomorphic, working in the pro-oligomorphic setting can be clearer since many concepts depend only on the topology and not $\Omega$.

Fix a pro-oligomorphic group $G$. An action of $G$ on a set $X$ is \defn{smooth} if every point in $X$ has open stabilizer in $G$, and \defn{finitary} if $G$ has finitely many orbits on $X$. We use the term ``$G$-set'' to mean ``set equipped with a finitary and smooth $G$-action.'' Let $\bS(G)$ be the category of such $G$-sets. We let $\bone$ denote the one-point $G$-set. An important property of $\bS(G)$ is that it is closed under finite products \cite[\S 2.3]{repst}, and therefore fiber products as well. This class of categories was studied and intrinsically characterized in \cite{bcat}.

\subsection{Measures}

Fix a pro-oligomorphic group $G$ and a field $k$. We require the notion of measure introduced in \cite{repst}.

\begin{definition} \label{defn:meas}
A $k$-valued \defn{measure} for $G$ is a rule assigning to each morphism $f \colon Y \to X$ of transitive $G$-sets a quantity $\mu(f)$ in $k$ such that:
\begin{enumerate}
\item If $f$ is an isomorphism then $\mu(f)=1$.
\item We have $\mu(g \circ f)=\mu(g) \circ \mu(f)$ when defined.
\item Let $f$ be as above, let $X' \to X$ be another morphism of transitive $G$-sets, and let $f' \colon Y' \to X'$ be the base change of $f$. Let $Y'=\bigsqcup_{i=1}^n Y'_i$ be the orbit decomposition of $Y'$, and let $f'_i$ be the restriction of $f'$ to $Y'_i$. Then $\mu(f)=\sum_{i=1}^n \mu(f'_i)$.
\end{enumerate}
\end{definition}

There is a universal measure $\mu^{\rm univ}$ valued in a ring $\Theta(G)$. To define $\Theta(G)$, start with the polynomial ring in symbols $[f]$, where $f$ runs over maps of transitive $G$-sets, and then quotient by the ideal generated by relations corresponding to the measure axioms. The universal measure is defined by $\mu^{\rm univ}(f)=[f]$. If $\mu$ is a $k$-valued measure then there is a unique ring homomorphism $\phi \colon \Theta(G) \to k$ such that $\mu=\phi \circ \mu^{\rm univ}$.

For a morphism $f \colon Y\to X$ of $G$-sets with $X$ transitive and $Y$ finitary, it will be convenient to define $[f]=\sum_i[f_i]$ in $\Theta(G)$, where the $f_i$ are the restrictions of $f$ to the orbits of $Y$. We similarly define $\mu(f)$, when $\mu$ is a measure. For a finitary $G$-set $Y$, we put $[Y]=[f]$ and $\mu(Y)=\mu(f)$, where $f \colon Y \to \bone$ is the unique map to the one-point set.

\subsection{Integration and matrices}

Fix a $k$-valued measure $\mu$ on $G$. Let $X$ be a $G$-set. A \defn{Schwartz function} on $X$ is a function $\phi \colon X \to k$ that is invariant under an open subgroup of $G$ and that has finitary support. We let $\cC(X)$ denote the Schwartz space of $X$, i.e., the $k$-vector space of all Schwartz functions. Given $\phi \in \cC(X)$, we define its integral
\begin{displaymath}
\int_X \phi(x) dx
\end{displaymath}
as in \cite{repst}. We note that this depends on the measure $\mu$, even though it is absent from the notation. Integration defines a $k$-linear map $\cC(X) \to k$. More generally, if $f \colon Y \to X$ is a map of finitary $G$-sets then the measure can be used to define a push-forward map $f_* \colon \cC(Y) \to \cC(X)$, see \cite{repst}. In the monoidal structure it will become dual to the ordinary pull-back map $f^*:\cC(X)\to\cC(Y)$. 

Let $X$ and $Y$ be finitary $G$-sets. A \defn{$Y \times X$ matrix} is simply a Schwartz function on $Y \times X$. Given a $Y \times X$ matrix $A$ and a $Z \times Y$ matrix $B$, we define their product $BA$ to be the $Z \times X$ matrix to be the push-forward of the function $(z,y,x) \mapsto B(z,y)A(y,x)$ under the projection $Z \times Y \times X \to Z \times X$. Matrix multiplication has all the expected properties. If $A$ is a $Y \times X$ matrix then $A$ defines a linear map $\cC(X) \to \cC(Y)$ by matrix multiplication, where we identify $\cC(X)$ with $X \times \bone$ matrices.

\subsection{The tensor category} \label{ss:uperm}

Let $G$ and $\mu$ be as above. In \cite[\S 8]{repst}, we defined a tensor category $\uPerm(G, \mu)$. We recall the main points of the definition. The objects of this category are labelled by the finitary $G$-sets. Since there exists a faithful functor from $\uPerm(G, \mu)$ to the category of vector spaces that sends the object labelled by $X$ to the corresponding Schwartz space $\cC(X)$, we actually write $\cC(X)$ also for the object in $\uPerm(G, \mu)$. A morphism $\cC(X) \to \cC(Y)$ is a $G$-invariant $Y \times X$ matrix, or, equivalently, the linear map defined by such a matrix. Composition is given by matrix multiplication, or, equivalently, composition of linear transformations. Direct sums and tensor products are defined on objects by
\begin{displaymath}
\cC(X) \oplus \cC(Y) = \cC(X \amalg Y), \qquad
\cC(X) \otimes \cC(Y) = \cC(X \times Y)
\end{displaymath}
and on morphisms using the usual constructions (block matrices and Kronecker products). We note that the vector space $\cC(X \times Y)$ is not the tensor product of the vector spaces $\cC(X)$ and $\cC(Y)$; in other words, the forgetful functor from $\uPerm(G, \mu)$ to vector spaces is not monoidal. The category $\uPerm(G, \mu)$ is rigid, and every object is self-dual. The dimension of $\cC(X)$ is given by the measure $\mu(X)$.

Suppose $f \colon Y \to X$ is a map of $G$-sets. We then have linear maps
\begin{displaymath}
f_* \colon \cC(Y) \to \cC(X), \qquad f^* \colon \cC(X) \to \cC(Y).
\end{displaymath}
These maps arise from matrices (the matrix is essentially the indicator function of the graph of $f$), and thus are maps in the category $\uPerm(G, \mu)$; see \cite[\S 7.7]{repst}. These maps generate all maps, in the following sense. Suppose $Z$ is an orbit on $Y \times X$, and let $A_Z$ denote its indicator function, thought of as a $Y \times X$ matrix; note that such matrices span the space $\Hom(\cC(X), \cC(Y))$ in the category $\uPerm(G, \mu)$. Let $p \colon Z \to Y$ and $q \colon Z \to X$ be the projection maps. Then it is not difficult to verify that $A_Z=p_*q^*$. See \cite[Proposition~7.22]{repst} for details.

Let $X$ be a $G$-set. We have a unique map $p \colon X \to \bone$ and a diagonal map $i \colon X \to X \times X$. The maps
\begin{displaymath}
p^* \colon \cC(\bone) \to \cC(X), \qquad i^* \colon \cC(X) \otimes \cC(X) \to \cC(X)
\end{displaymath}
give $\cC(X)$ the structure of a commutative algebra object in $\uPerm(G, \mu)$. Letting $\delta_x \in \cC(X)$ denote the point mass at $x \in X$, the multiplication is given explicitly by $\delta_x \delta_y = 0$ if $x \ne y$ and $\delta_x^2=\delta_x$, i.e., the $\delta_x$ are orthogonal idempotents. If $f \colon Y \to X$ is a map of $G$-sets then the map $f^* \colon \cC(X) \to \cC(Y)$ is an algebra homomorphism.

\subsection{\texorpdfstring{$\Theta$}{Theta}-generators} \label{ss:theta-gen}

Fix a pro-oligomorphic group $G$. Let $\Sigma$ be the class of morphisms $f \colon Y \to X$ in $\bS(G)$, with $X$ transitive. By definition, the elements $[f]$ with $f \in \Sigma$ generate $\Theta(G)$. We now isolate a class of subsets $S \subset \Sigma$, called $\Theta$-generating sets, that have the property that the elements $[f]$ with $f \in S$ generate $\Theta(G)$. In fact, the property of being a $\Theta$-generating set is stronger than the property that its elements generate the ring $\Theta(G)$. Having a $\Theta$-generating set will be useful when we discuss mapping properties for $\uPerm(G, \mu)$.

Let $\Pi$ be a subclass of $\Sigma$. We say that $\Pi$ is a \defn{$\Theta$-class} if the following conditions hold:
\begin{enumerate}
\item $\Pi$ contains all isomorphisms of transitive $G$-sets.
\item If $g \colon Z \to Y$ and $f \colon Y \to X$ belong to $\Pi$ then so does $f \circ g$.
\item Let $f \colon Y \to X$ belong to $\Sigma$, let $X' \to X$ be a map of transitive $G$-sets, and let $f' \colon Y' \to X'$ be the base change of $f$. If $f$ belongs to $\Pi$ then so does $f'$.
\item Let $f \colon Y \to X$ be a map in $\Sigma$ and suppose that $Y=Y_1 \sqcup Y_2$. Let $f_i$ be the restriction of $f$ to $Y_i$. Then if any two of $f$, $f_1$, and $f_2$ belong to $\Pi$, so does the third.
\end{enumerate}
An arbitrary intersection of $\Theta$-classes is again a $\Theta$-class. It follows that if $S$ is any subset of $\Sigma$ then there is a unique minimal $\Theta$-class $\Pi$ containing $S$, and we say that $\Pi$ is \defn{$\Theta$-generated} by $S$. If $\Pi=\Sigma$ then we say that $S$ is a \defn{$\Theta$-generating set} for $G$. 

\begin{proposition} \label{prop:theta-gen}
Suppose that $S$ is a $\Theta$-generating set for $G$. Then the elements $[f]$, with $f \in S$, generate $\Theta(G)$ as a ring.
\end{proposition}

\begin{proof}
Let $R$ be the subring of $\Theta(G)$ generated by the elements $[f]$ with $f \in S$. Let $\Pi \subset \Sigma$ be the set of all $f$'s such that $[f] \in R$. We claim that $\Pi$ is a $\Theta$-class. We verify the axioms, using notation as used in the statements of the axioms. 

\begin{enumerate}
\item If $f$ is an isomorphism then $[f]=1$, and so $[f] \in R$, and so $f \in \Pi$.
\item If $g,f \in \Pi$ then $[gf]=[g] \cdot [f]$ belongs to $R$, and so $gf \in \Pi$.
\item We have $[f]=[f']$, and so $f \in \Pi$ if and only if $f' \in \Pi$.
\item We have $[f]=[f_1]+[f_2]$, and so if two of $[f]$, $[f_1]$, and $[f_2]$ belong to $R$ then so does the third. Thus if two of $f$, $f_1$, and $f_2$ belong to $\Pi$ then so does the third.
\end{enumerate}
Since $\Pi$ clearly contains $S$, it follows that $\Pi=\Sigma$ since $S$ is $\Theta$-generating. Thus $R$ contains $[f]$ for all $f \in \Sigma$, and thus $R=\Theta(G)$, as required.
\end{proof}

\begin{corollary} \label{cor:theta-gen}
Suppose that $S$ is a $\Theta$-generating set for $G$, and let $\mu$ and $\nu$ be two $k$-valued measures for $G$. If $\mu(f)=\nu(f)$ for all $f \in S$ then $\mu=\nu$.
\end{corollary}

\begin{proof}
Indeed, a $k$-valued measure is a ring homomorphism $\Theta(G) \to k$, and if two homomorphisms agree on generators then they are equal.
\end{proof}

\begin{remark} \label{rmk:weak-theta}
In the definition of $\Theta$-class, one can alter axiom (c) to ``$f \in \Pi$ if and only if $f' \in \Pi$.'' This leads to a stronger notion of $\Theta$-class, and a weaker notion of $\Theta$-generators. Some of our results work with this variant definition, and some do not. For example, Proposition~\ref{prop:theta-gen} does work: if $S$ is a ``weak $\Theta$-generating set'' then the classes $[f]$ with $f \in S$ generate $\Theta(G)$.
\end{remark}

\begin{remark}
The measure axioms do not use subtraction, and so one can define the notion of measure valued in a semi-ring. There is again a universal measure valued in a semi-ring version of $\Theta$. The above argument shows that a set of $\Theta$-generators will generate this universal semi-ring.
\end{remark}

\subsection{The relative case} \label{ss:relative}

Many of the constructions and definitions given above apply to certain subcategories of $\bS(G)$, and this additional generality leads to some important examples. To define these subcategories, we introduce a piece of terminology. A \defn{stabilizer class} in $G$ is a collection $\sE$ of open subgroups of $G$ satisfying the following conditions: (a) $\sE$ contains $G$; (b) $\sE$ is closed under finite intersections; (c) $\sE$ is closed under conjugation; and (d) $\sE$ forms a neighborhood basis of the identity of $G$, that is, every open subgroup of $G$ contains some member of $\sE$ as a subgroup.

Let $\sE$ be a stabilizer class. We say that a $G$-set $X$ is \defn{$\sE$-smooth} if the stabilizer of any element of $X$ belongs to $\sE$. We write $\bS(G, \sE)$ for the full subcategory of $\bS(G)$ spanned by $\sE$-smooth $G$-sets. This is closed under products, fiber products, and disjoint unions, but not under quotients (in general). A \defn{measure} for $G$ relative to $\sE$ is a rule assigning to each morphism $f \colon Y \to X$ of transitive $\sE$-smooth $G$-sets a quantity $\mu(f)$ such that the obvious analogs of the usual axioms hold. A measure $\mu$ gives rise to a tensor category $\uPerm(G, \sE; \mu)$, the objects of which are the Schwartz spaces $\cC(X)$ where $X$ is an object of $\bS(G, \sE)$. There is also a natural notion of $\Theta$-generators for $G$ relative to $\sE$.

If $\Omega$ is a $G$-set then we obtain a stabilizer class $\sE(\Omega)$ by taking all subgroups of $G$ that occur as the stabilizer of some element in $\Omega^n$, for some $n$. A transitive $G$-set is $\sE(\Omega)$-smooth if and only if it is isomorphic to an orbit on some power of $\Omega$.

\section{\'Etale algebras in tensor categories} \label{s:etale}

In this section, we examine \'etale algebras in tensor categories. We begin by reviewing some fairly standard results, though we include proofs since we do not know a good reference for our level of generality. In \S \ref{ss:gamma}, we introduce the notion of a \defn{uniform map} of \'etale algebras, and attach to such maps a numerical invariant $\gamma$. This concept is used to define the notion of compatibility of a functor and measure in our general mapping property; see Definition~\ref{defn:compatible}. In \S \ref{ss:lextensive} we show that (the opposite of) the category of \'etale algebras is lextensive, which informs our approach to universal properties for categories of the form $\bS(G)$; see, e.g., Theorem~\ref{thm:comb}.

We fix a Karoubian tensor category $\fT$ for the duration of \S \ref{s:etale}.

\subsection{\'Etale algebras} \label{ss:etale}

By an ``algebra'' in $\fT$ we mean a commutative, associative, unital algebra, and by a ``rigid algebra,'' we mean an algebra that is also a rigid object. For an algebra $A$ we denote multiplication by $m=m_A$ and its unit by $\eta=\eta_A$.  If $A$ is a rigid algebra, we have the trace map $\epsilon_{A/\bbone} \colon A \to \bbone$. When no confusion is possible as to which category $A$ is considered in, we write $\epsilon_A=\epsilon_{A/\bbone}$. The trace induces a trace pairing $A \otimes A \to \bbone$, via $(x,y) \mapsto \epsilon(xy)$. We say that a rigid algebra $A$ is \defn{\'etale} if the trace pairing is perfect. See \cite[\S 4.1]{discrete} for background.

An important property of \'etale algebras is that the multiplication map $A \otimes A \to A$ has a unique splitting $s:A\to A\otimes A$ as $A \otimes A$-modules. This means that there is a unique idempotent $\sigma=\sigma_A$ in the $k$-algebra $\Gamma(A \otimes A)$ that satisfies $(x \otimes 1)\sigma=(1 \otimes x) \sigma$ and $m_A(\sigma)=1$, or, equivalently,
\begin{displaymath}
m_{A\otimes A}\circ (A\otimes\eta_A\otimes\sigma)\;=\;m_{A\otimes A}\circ (\eta_A\otimes A\otimes\sigma),
\end{displaymath}
 and $m\circ\sigma=\eta$.

\begin{example}
Consider the category $\uPerm(G, \mu)^{\rm kar}$ associated to an oligomorphic group $G$ and a measure $\mu$. Let $X$ be a $G$-set, and let $p \colon X \to \bone$ be the unique map. We have seen (\S \ref{ss:uperm}) that $\cC(X)$ is naturally a commutative algebra. It is not difficult to show that $p_* \colon \cC(X) \to \cC(\bone)=\bbone$ is the trace map for $\cC(X)$, and that the trace pairing is non-degenerate (see \cite[\S 8.4]{repst}). Thus $\cC(X)$ is an \'etale algebra.
\end{example}

\subsection{Modules}

Fix an \'etale algebra $A$ in $\fT$. We now consider the category $\Mod_A$ of $A$-modules in $\fT$.

\begin{proposition}\label{Prop:Summ}
For any $A$-module $M$, the natural action map $a:A \otimes M \to M$ is a split epimorphism in $\Mod_A$.
\end{proposition}

\begin{proof}
We define the composite morphism
$$M\xrightarrow{\eta\otimes M}A\otimes M\xrightarrow{s\otimes M}A\otimes A\otimes M\xrightarrow{A\otimes a}A\otimes M.$$
It composes to the identity with $a$ by associativity of $a$, and it is an $A$-module morphism by associativity and the fact that $s$ is a morphism of bimodules.
\end{proof}

\begin{corollary}
The module category $\Mod_A$ is a tensor category with tensor product given by the co-equalizer of the two action morphisms
\begin{displaymath}
M \otimes A \otimes N \rightrightarrows M \otimes N \to M \otimes_A N.
\end{displaymath}
\end{corollary}

\begin{proof}
It is a standard fact that, for any algebra $A$ in $\fT$, the category of free $A$-modules $A\otimes X$, with $X\in\fT$, is a tensor category with the above tensor product. Indeed, using the standard splitting of the bar complex by chain homotopy $\eta\otimes A^{\otimes n}:A^{\otimes n}\to A^{\otimes n+1}$, shows in particular that 
$(A\otimes X)\otimes_A(A\otimes Y)$ is given by $A\otimes X\otimes Y$. The tensor product extends immediately to direct summands of free modules, and thus to all of $\Mod_A$, by Proposition~\ref{Prop:Summ}.
\end{proof}

We denote the defining bi-natural epimorphism by $\pi_{M,N}:M\otimes N\tto M\otimes_AN$. For free modules we have a canonical section
$$(A\otimes X)\otimes_A(A\otimes Y)\xrightarrow{\sim} A\otimes X\otimes Y \xrightarrow{s\otimes X\otimes Y}A\otimes A\otimes X\otimes Y\xrightarrow{\sim}(A\otimes X)\otimes (A\otimes Y),$$
which by considering direct summands extends to a bi-natural morphism $s_{M,N}:M\otimes_A N\to M\otimes N$ such that $\pi_{M,N}\circ s_{M,N}=\id_{M\otimes_AN}$ and
\begin{equation}\label{eq:ps}(M\otimes \pi_{N,P})\circ (s_{M,N}\otimes P)=(s_{M,N}\otimes_A P)\circ (M\otimes_A \pi_{N,P}),\end{equation}
for $A$-modules $M,N,P$, which can be proved again by reducing to free modules. Furthermore, the left unitor of the monoidal structure and its inverse are given by
\begin{equation}\label{eq:unitor}A\otimes_AM\xrightarrow{s_{A,M}}A\otimes M\xrightarrow{\epsilon\otimes M}M\quad\mbox{and}\quad M\xrightarrow{\eta\otimes M}A\otimes M\xrightarrow{\pi_{A,M}}A\otimes_AM,\end{equation}
as follows for instance from \cite[Proposition~4.11(b)]{discrete}.

\begin{proposition} \label{prop:rigid-module}
An $A$-module $M$ is rigid in $\Mod_A$ if and only if it is rigid in $\fT$, and the underlying object of the dual in $\Mod_A$ is the dual in $\fT$. 
\end{proposition}

\begin{proof}
If $M$ is rigid in $\fT$ then $A \otimes M$ is rigid in $\Mod_A$, and thus so is $M$, being a summand, by Proposition~\ref{Prop:Summ}. We now prove the converse.

Let $M^{\vee}$ be the dual of $M$ in $\Mod_A$ and let
\begin{displaymath}
\alpha \colon A \to M \otimes_A M^{\vee}, \qquad \beta \colon M^{\vee} \otimes_A M \to A
\end{displaymath}
be the co-evaluation and evaluation maps. Then we define 
\begin{displaymath}
\alpha' \colon \bbone \to M \otimes M^{\vee}, \qquad \beta' \colon M^{\vee} \otimes M \to \bbone
\end{displaymath}
by $\alpha'=s_{M,M^\vee}\circ \alpha\circ\eta$ and $\beta'=\epsilon\circ \beta\circ \pi_{M^\vee,M}$. The relation 
$$(M\otimes\beta')\circ (\alpha'\otimes M)=\id_M$$
then follows by first applying \eqref{eq:ps}, then using naturality of $\pi$ and $s$, and finally applying the description of the unitor in \eqref{eq:unitor} to reduce to the corresponding relation for $\alpha,\beta$. The second relation is proved identically.
\end{proof}

\begin{corollary}
If $\fT$ is rigid then so is $\Mod_A$.
\end{corollary}

Suppose $M$ is a rigid $A$-module, or equivalently $M$ is rigid in $\fT$. We can define the internal endomorphism algebra of $M$ in $\Mod_A$ as the usual equalizer
$$\ul{\End}_A(M)\to\ul{\End}(M)\rightrightarrows \ul{\Hom}(A\otimes M,M).$$
Here $\ul{\Hom}(X,Y)$, for $X$ rigid is simply $X^\vee\otimes Y$.
The equalizer exists since $A$ is an \'etale algebra, and moreover coincides with the quotient definition $M^{\vee} \otimes_A M$ (the internal endomorphism algebra of $M$ in $\Mod_A$), since both correspond to $\sigma(M^\vee\otimes M)$. There are trace maps $\tr_A=\epsilon_{\ul{\End}_A(M)/A} $ and $\tr=\epsilon_{\ul{\End}(M)/\bbone} $, and also a forgetful map $\ul{\End}_A(M) \to \ul{\End}(M)$. We now examine how these relate. 

\begin{proposition} \label{prop:trace-change-ring}
Let $M$ be a rigid $A$-module. Then the following diagram commutes
\begin{displaymath}
\xymatrix{
\ul{\End}_A(M) \ar[r] \ar[d]_{\tr_A} & \ul{\End}(M) \ar[d]^{\tr} \\
A \ar[r]^{\epsilon_A} & \bbone }
\end{displaymath}
\end{proposition}
\begin{proof}
We take the definition of $\ul{\End}_A(M)$ as $M^{\vee} \otimes_A M$. In this definition, the forgetful map becomes $s_{M^\vee,M}$. Then $\tr_A$ is the evaluation $\beta$ of $M$ as a rigid object in $\Mod_A$ and $\tr$ is the evaluation $\beta'$ of $M$ as a rigid object in $\fT$. For convenience we chose the same symbols as in the proof of Proposition~\ref{prop:rigid-module}. The identity $\epsilon_A\circ \beta=\beta'\circ s_{M^\vee,M}$ then follows immediately from the definition of $\beta'$ in that proof.
\end{proof}

If $f$ is an endomorphism of $M$ then we can form its trace $\tr_A(f)$ in the category $\Mod_A$, which belongs to $\Gamma(A)$. We now compare this with the trace $\tr(f)$ of $f$ computed in $\fT$.

\begin{corollary}
Let $M$ be a rigid $A$-module and let $f$ be an $A$-module endomorphism of $M$. Then $\epsilon_A(\tr_A(f))=\tr(f)$.
\end{corollary}

\begin{proof}
This follows from the proposition upon taking invariants.
\end{proof}

\begin{corollary}
If $\tr_A(f) \in k$ then $\tr(f)=\dim(A) \tr_A(f)$.
\end{corollary}

As usual, an algebra in $\Mod_A$ is the same thing as an algebra in $\fT$ equipped with an algebra homomorphism from $A$. We now examine the \'etale condition for such algebras.

\begin{proposition} \label{prop:etale-in-mod}
Let $A \to B$ be an algebra homomorphism in $\fT$. Then $B$ is \'etale in $\Mod_A$ if and only if $B$ is \'etale in $\fT$.
\end{proposition}

\begin{proof}
This is proved in \cite[Proposition~5.10]{discrete} in case $\fT$ is pre-Tannakian. However, the same proof now applies in general thanks to Proposition~\ref{prop:rigid-module}, which replaces the need for \cite[Proposition~5.1]{discrete} in the proof.
\end{proof}

\subsection{Duality} \label{ss:dual}

Let $f \colon A \to B$ be a map of \'etale algebras. We can then regard $B$ as an \'etale algebra in the tensor category $\Mod_A$, so that we have
the trace $\epsilon_{B/A} \colon B \to A$ for $B$ in $\Mod_A$. Proposition~\ref{prop:trace-change-ring}, applied to the map $B \to \ul{\End}_A(B)$, implies we have a transitive law for traces, i.e., $\epsilon_B = \epsilon_A \circ \epsilon_{B/A}$. Since $\epsilon_{B/A}$ is a morphism in $\Mod_A$, it is $A$-linear by definition. We thus have the identity
\begin{displaymath}
\epsilon_B(f(a)b) = \epsilon_B(\epsilon_{B/A}(f(a)b)) = \epsilon_A(a \epsilon_{B/A}(b)).
\end{displaymath}
This identity is really one of maps $A \otimes B \to \bbone$. This shows that $\epsilon_{B/A}$ is the dual to the map $A \to B$ in $\fT$, where $A$ and $B$ are identified with their own duals via their trace pairings. We therefore sometimes write $f^{\vee}$ in place of $\epsilon_{B/A}$. 

The following result is very helpful, as it often allows us to reduce to the case where $A=\bbone$.

\begin{proposition}
Let $A \to B$ and $f \colon B \to C$ be maps of \'etale algebras. Then $f^{\vee}$ is the same whether computed in $\fT$ or in $\Mod_A$.
\end{proposition}

\begin{proof}
Let $f^{\vee}$ be the dual of $f$ in $\fT$, and let $f_A^{\vee}$ be the dual in $\Mod_A$. By definition, $f^{\vee} \colon C \to B$ is the unique map satisfying
\begin{displaymath}
\epsilon_B(b f^{\vee}(c)) = \epsilon_C(f(b) c),
\end{displaymath}
i.e., the two sides agree as maps $B \otimes C \to \bbone$. Since $f$ and $f^{\vee}$ are both $A$-linear, the two maps above actually define maps $B \otimes_A C \to \bbone$. Similarly, $f^{\vee}_A$ is the unique map satisfying
\begin{displaymath}
\epsilon_{B/A}(b f^{\vee}_A(c)) = \epsilon_{C/A}(f(b) c),
\end{displaymath}
i.e., the two maps $B \otimes_A C \to A$ agree. Applying $\epsilon_A$ to the second equation, we see that $f^{\vee}_A$ satisfies the defining property of $f^{\vee}$, and so the two coincide.
\end{proof}

\begin{proposition}
Let $f \colon B \to C$ be a map of \'etale algebras, let $A$ be another \'etale algebra, and consider the map
\begin{displaymath}
f \otimes \id \colon B \otimes A \to C \otimes A.
\end{displaymath}
Then $(f \otimes \id)^{\vee}=f^{\vee} \otimes \id$.
\end{proposition}

\begin{proof}
Consider the tensor functor $\fT \to \Mod_A$ given by tensoring with $A$. Since formation of dual maps is compatible with tensor functors, we see that $f^{\vee} \otimes \id$ is the dual of $f \otimes \id$ computed in $\Mod_A$, which we have seen is the same as the dual computed in $\fT$.
\end{proof}

\begin{proposition} \label{prop:etale-inv}
If $f \colon A \to B$ is an isomorphism of \'etale algebras then $f^{\vee} \colon B \to A$ is the inverse of $f$.
\end{proposition}

\begin{proof}
This is clear if $A=\bbone$, and follows in general by passing to $\Mod_A$.
\end{proof}

\begin{proposition} \label{prop:etale-bc}
Given a cartesian square
\begin{displaymath}
\xymatrix{
B \ar[r]^{g'} & B' \\
A \ar[r]^g \ar[u]^f & A' \ar[u]_{f'} }
\end{displaymath}
of \'etale algebras, the diagram
\begin{displaymath}
\xymatrix{
B \ar[r]^{g'} \ar[d]_{f^{\vee}} & B' \ar[d]^{(f')^{\vee}} \\
A \ar[r]^g & A' }
\end{displaymath}
commutes.
\end{proposition}

\begin{proof}
First suppose that $A=\bbone$. Then $B'=B \otimes A'$ and $g'=\id_B \otimes g$ and $f'=f \otimes \id_{A'}$. Since $(g')^{\vee}=\id_B \otimes g^{\vee}$ and $(f')^{\vee}=f^{\vee} \otimes \id_B$, the result follows. The general case now follows upon passing to $\Mod_A$.
\end{proof}

\subsection{Uniform maps} \label{ss:gamma}

Let $\fT$ be a tensor category. Suppose that $f \colon A \to B$ is a map of \'etale algebras. The map $f^{\vee} \colon B \to A$ induces a map $ \Gamma(B) \to \Gamma(A)$. We define $\tilde{\gamma}(f)$ to be the element $f^{\vee}(1)=f^\vee\circ\eta_B$ of $\Gamma(A)$. We say that $f$ is \defn{uniform} if $A$ is non-zero and $\tilde{\gamma}(f)$ belongs to $k=k\cdot\eta_A \subset \Gamma(A)$. In this case, we let $\gamma(f)=\tilde{\gamma}(f)$, regarded as an element of $k$. If $f$ is an isomorphism then $f$ is uniform with $\gamma(f)=1$ (Proposition~\ref{prop:etale-inv}). We now give some examples and basic properties of this construction.

\textit{(a) Vector spaces.} Suppose $\fT$ is the category of finite dimensional complex vector spaces and let $f \colon A \to B$ be a map of non-zero \'etale algebras. Let $X=\operatorname{Spec}(A)$ and $Y=\operatorname{Spec}(B)$, which can be regarded simply as finite sets, and let $Y \to X$ be the map induced by $f$. Then $\tilde{\gamma}(f)$ is the function on $X$ that assigns to a point the cardinality of its fiber. Thus $f$ is uniform if and only if all fibers of $Y \to X$ have the same cardinality; in this case, $\gamma(f)$ is this common cardinality. The situation is similar for oligomorphic tensor categories:

\textit{(b) Oligomorphic groups.} Suppose $\fT=\uPerm(G, \mu)$ for some oligomorphic group $G$ with measure $\mu$, and let $f \colon Y \to X$ be a map of $G$-sets with $X$ transitive. Then $f^* \colon \cC(X) \to \cC(Y)$ is uniform with $\gamma(f^*)=\mu(f)$ by \cite[Proposition~7.21]{repst}.

\textit{(c) Connection to dimension.} Let $B$ be an \'etale algebra and let $f \colon \bbone \to B$ be the unit. Then $\tilde{\gamma}(f)=\dim(B)$. Indeed, we have $f^\vee=\epsilon$ and thus $\tilde{\gamma}(f)=\epsilon\circ\eta$, so that the conclusion follows by definition of the trace $\epsilon$ in \cite[\S 4.1]{discrete}. More generally, if $f \colon A \to B$ is any map of \'etale algebras then $\tilde{\gamma}(f)=\dim_A{B}$. In particular, if $\fT$ is pre-Tannakian and $k$ is algebraically closed and $A=\bigoplus_{i=1}^n A_i$ is the decomposition of $A$ into simple \'etale algebras, and $B=\bigoplus_{i=1}^n B_i$ is the corresponding decomposition of $B$, then $f$ is uniform if and only if $\dim_{A_i}(B_i)$ is independent of $i$ (and $A \ne 0$).

\textit{(d) Pullbacks.} Given a cartesian square as in Proposition~\ref{prop:etale-bc}, we have $g(\tilde{\gamma}(f))=\tilde{\gamma}(f')$. In particular, if $f$ is uniform then so is $f'$, and $\gamma(f)=\gamma(f')$. These claims follows immediately from Proposition~\ref{prop:etale-bc}.

\textit{(e) Composition.} Let $f \colon A \to B$ and $g \colon B \to C$ be maps of \'etale algebras. Then
\begin{displaymath}
\tilde{\gamma}(gf) = f^{\vee}(\tilde{\gamma}(g)).
\end{displaymath}
Indeed, we have $(gf)^{\vee}=f^{\vee} \circ g^{\vee}$, so evaluating at~1 yields the equation. In particular, if $f$ and $g$ are uniform then so is $gf$, and
\begin{displaymath}
\gamma(gf) = \gamma(g) \cdot \gamma(f).
\end{displaymath}
Indeed, simply observe that $f^{\vee}$ is $k$-linear, so $\tilde{\gamma}(g)=\gamma(g)$ pulls out of it.

\textit{(f) Addition.} Let $f \colon A \to B_1 \oplus B_2$ be a map of \'etale algebras, and let $f_i \colon A \to B_i$ be the projection of $f$. Then one easily sees that
\begin{displaymath}
\tilde{\gamma}(f)=\tilde{\gamma}(f_1)+\tilde{\gamma}(f_2).
\end{displaymath}
In particular, if two of $f$, $f_1$, and $f_2$ are uniform then so is the third, and the above relation holds with $\gamma$ in place of $\tilde{\gamma}$.

\textit{(g) Functoriality.} Let $\Phi \colon \fT \to \fT'$ be a tensor functor and let $f \colon A \to B$ be a map of \'etale algebras in $\fT$. Then $\tilde{\gamma}(\Phi(f))=\Phi(\tilde{\gamma}(f))$. In particular, if $f$ is uniform and $\Phi(A)$ is non-zero then $\Phi(f)$ is uniform and $\gamma(\Phi(f))=\gamma(f)$.

\subsection{The category of \'etale algebras} \label{ss:lextensive}

Let $\Et(\fT)$ be the category of \'etale algebras in $\fT$, where morphisms are algebra homomorphisms. We now investigate the structure of this category.

Let $\cS$ be a category. We say that $\cS$ \defn{extensive}\footnote{This is sometimes called ``finitely extensive.''} if it has finite co-products, and for any objects $X$ and $Y$ the functor
\begin{displaymath}
\cS_{/X} \times \cS_{/Y} \to \cS_{/(X \amalg Y)}, \qquad
(A,B) \mapsto A \amalg B
\end{displaymath}
is an equivalence; here $\cS_{/X}$ denotes the category of objects over $X$. We say that $\cS$ is \defn{lextensive} if it is extensive and also has finite limits. Given a subobject $Y \subset X$, a \defn{complement} is a subobject $Y'$ such that the natural map $Y \amalg Y' \to X$ is an isomorphism. It is easy to see that complements are unique when they exist. We say that $\cS$ \defn{has complements} if every subobject has a complement. We note that for any pro-oligomorphic group, the category $\bS(G)$ is lextensive and has complements. In fact, this is true for the categories $\bS(G, \sE)$ associated to a stabilizer class $\sE$ as well.

The following is our main result on the structure of $\Et(\fT)$.

%There are many examples of lextensive categories:
%\begin{enumerate}
%\item For any pro-oligomorphic group $H$, the category $\bS(H)$ is lextensive. We note that in \cite{bcat}, Harman and the second author called categories of this form \defn{pre-Galois}, and gave a purely category theoretic characterization of them.
%\item In \cite[\S 2.6]{repst}, we introduced the notion of a \defn{stabilizer class} $\sE$ in a pro-oligomorphic group $H$. The category $\bS(H, \sE)$ of finitary $\sE$-smooth $H$-sets is lextensive.
%\item We will see below that the (opposite of the) category of \'etale algebras in any tensor category is lextensive.
%\item Any (elementary or Grothendieck) topos is lextensive.
%\end{enumerate}
%Our main motivation for working in the class of lextensive categories is that it includes examples (a), (b), and (c).

\begin{proposition} \label{prop:etale-cat}
The category $\Et(\fT)^{\op}$ is lextensive and has complements. Moreover, if $\Phi \colon \fT \to \fT'$ is a tensor functor (with $\fT'$ Karoubian), then the induced functor $\Phi \colon \Et(\fT)^{\op} \to \Et(\fT')^{\op}$ is additive and left-exact.
\end{proposition}

The proposition essentially says that $\Et(\fT)^{\op}$ behaves like the category of finite sets in some important ways. This is most clearly illustrated by considering case where $\fT$ is the category of representations of finite group $G$ over an algebraically closed field: in this case $\Et(\fT)^{\op}$ is equivalent to the category $\bS(G)$ of finite $G$-sets.

We break the proof into a few lemmas.

\begin{lemma}
The category $\Et(\fT)^{\op}$ has finite limits and finite co-products.
\end{lemma}

\begin{proof}
If $A$ and $B$ are \'etale algebras then the product algebra $A \oplus B$ is \'etale, and is the categorical product in $\Et(\fT)$. If $A \to B$ and $A \to C$ are maps of \'etale algebras then $B\otimes C$ is an \'etale algebra. Moreover, the algebra $B\otimes_AC$ is then a factor algebra, and thus \'etale by \cite[Proposition~4.1]{discrete}. To see that it is a factor algebra, observe that we can realise $B\otimes_AC$ as $\sigma(B\otimes C)$, where the idempotent $\sigma$ remains an idempotent in $\Gamma(B\otimes C)$. As is well-known, the tensor product is the categorical push-out in the category of algebras, so certainly in $\Et(\fT)$. Moreover, the unit object $\bbone$ is \'etale, and the initial object of $\Et(\fT)$. We thus see that $\Et(\fT)^{\op}$ has finite co-products and finite limits. 
\end{proof}

\begin{lemma}
The category $\Et(\fT)^{\op}$ is extensive.
\end{lemma}

\begin{proof}
If $A$ and $B$ are algebras then $\Mod_{A \oplus B}$ is equivalent to $\Mod_A \oplus \Mod_B$; if they are \'etale algebras this is an equivalence of tensor categories. From this, it follows easily that $\Et(\fT)^{\op}$ is extensive.
\end{proof}

\begin{lemma}
Suppose that $\bbone \to A$ is an epimorphism of \'etale algebras. Then $A$ is a direct factor of $\bbone$.
\end{lemma}

\begin{proof}
Since $\bbone \to A$ is an epimorphism, it follows that the map $A \to A \otimes A$ given by $x \mapsto 1 \otimes x$ is an isomorphism. Indeed, it is both an epimorphism and a split monomorphism, alternatively we can pass to the opposite category and use the standard fact that subterminal objects $X$ satisfy $X\times X=X$. We thus see that $\dim(A)=\dim(A)^2$, and so $e=\dim(A)$ is an idempotent of $\Gamma(\bbone)$. This idempotent decomposes $\fT$ as $\fT_1 \oplus \fT_2$, where $\fT_1$ is the category of $e \bbone$ modules and $\fT_2$ is the category of $(1-e) \bbone$ modules. Now, on the other hand, we have
\begin{displaymath}
1 = \dim_A(A) = \dim_A(A \otimes A),
\end{displaymath}
since the isomorphism $A \to A \otimes A$ is one of $A$-modules. By base change (\S \ref{ss:gamma}(d)), $\dim_A(A \otimes A)$ is the image of $\dim(A)$ under $\Gamma(\bbone) \to \Gamma(A)$. We thus see that the map $\Gamma(\bbone) \to \Gamma(A)$ sends $e$ to~1, and therefore $1-e$ to~0. This shows that $A$ lives in the category $\fT_1$. We may thus replace $\fT$ with $\fT_1$, and thereby assume $e=1$, i.e., $\dim(A)=1$. This means that the composition $\bbone \to A \to \bbone$ is the identity, where the first map is the unit and the second its dual, and so $A=\bbone \oplus X$ for some object $X$ of $\fT$. Since the map $A \to A \otimes A$ is an isomorphism and sends $\bbone$ to $\bbone\otimes\bbone$ and $X$ to $\bbone\otimes X$, it follows that $X\otimes\bbone =X=0$. This completes the proof.
\end{proof}

\begin{lemma}
Suppose that $A \to B$ is an epimorphism of \'etale algebras. Then $B$ is a direct factor of $A$.
\end{lemma}

\begin{proof}
This follows from the previous lemma upon passing to $\Mod_A$. Note that $B$ is an \'etale algebra in $\Mod_A$ by Proposition~\ref{prop:etale-in-mod}, and the morphism $A \to B$ remains an epimorphism in $\Et(\Mod_A)$, again by Proposition~\ref{prop:etale-in-mod}.
\end{proof}

The lemma implies that monomorphisms in $\Et(\fT)^{\op}$ admit complements. The following lemma thus completes the proof of the proposition.

\begin{lemma}
if $\Phi \colon \fT \to \fT'$ is a tensor functor (with $\fT'$ Karoubian), then the induced functor $\Phi^{\mathrm{et}}=\Phi \colon \Et(\fT)^{\op} \to \Et(\fT')^{\op}$ is additive and left-exact.
\end{lemma}

\begin{proof}
That $\Phi^{\mathrm{et}}$ is additive follows from the fact that $\Phi$ is additive. Since $\Phi$ is monoidal, it follows that $\Phi^{\mathrm{et}}$ preserves products and the final object. Since $B \otimes_A C$ is canonically a summand of $B \otimes C$, it follows that $\Phi^{\mathrm{et}}$ preserves fiber products. 
\end{proof}

\begin{remark}
If $\fT$ is pre-Tannakian then \cite[Theorem~6.1]{discrete} shows that $\Et(\fT)^{\op}$ is pre-Galois, that is, of the form $\bS(G)$ for some pro-oligomorphic group $G$. When $\fT$ is not pre-Tannakian, $\Et(\fT)^{\op}$ need not be pre-Galois. For instance, it is possible to get categories of the form $\bS(G, \sE)$ with non-trivial stabilizer class $\sE$. It is also possible to get categories in which objects need not admit a finite decomposition into atomic objects. We do not know exactly how ``bad'' $\Et(\fT)^{\op}$ can be in general.
\end{remark}

\section{Universal properties of oligomorphic tensor categories} \label{s:genmap}

In this section we establish a universal property for oligomorphic tensor categories. It states that tensor functors $\Phi \colon \uPerm(G, \mu) \to \fT$ correspond to certain kinds of functors $\Psi \colon \bS(G) \to \Et(\fT)^{\op}$. After proving this theorem, we establish a number of auxiliary results which aid in applying it. 

\subsection{The main theorem}

Let $G$ be a pro-oligomorphic group equipped with a $k$-valued measure $\mu$, and put $\fP = \uPerm(G, \mu)$. Let $\fT$ be a Karoubian tensor category. We aim to describe the category $\Fun^{\otimes}(\fP, \fT)$ of tensor functors.

Suppose we have a tensor functor $\Phi \colon \fP \to \fT$. Let $\Psi^{\circ}$ be the composition
\begin{displaymath}
\xymatrix{
\bS(G)^{\op} \ar[r] & \Et(\fP) \ar[r]^{\Phi} & \Et(\fT), }
\end{displaymath}
where the first functor maps a $G$-set $X$ to the \'etale algebra $\cC(X)$, and let
\begin{displaymath}
\Psi \colon \bS(G) \to \Et(\fT)^{\op}
\end{displaymath}
be the opposite functor to $\Psi$. Since the functor $\bS(G) \to \Et(\fP)^{\op}$ is additive and left-exact, it follows from Proposition~\ref{prop:etale-cat} that $\Psi$ is additive and left-exact.

Suppose now that $f \colon Y \to X$ is a map of $G$-sets with $X$ transitive. We have seen that the algebra homomorphism $f^* \colon \cC(X) \to \cC(Y)$ is uniform with $\gamma(f^*)=\mu(f)$ (\S \ref{ss:gamma}(b)). Assuming $\Phi(\cC(X))$ is non-zero, it follows that $\Phi(f^*)$ is also uniform with $\gamma(\Phi(f^*))=\mu(f)$ (\S \ref{ss:gamma}(g)). This suggests the following definition:

\begin{definition} \label{defn:compatible}
Let $\Psi \colon \bS(G) \to \Et(\fT)^{\op}$ be a functor. We say that $\Psi$ is \defn{compatible} with $\mu$ if for every map $f \colon Y \to X$ in $\bS(G)$ with $X$ transitive, either (a) $\Psi(X)=0$; or (b) $\Psi(f)$ is a uniform map of \'etale algebras with $\gamma(\Psi(f))=\mu(f)$.
\end{definition}

We are now ready for the main theorem. For a category $\fX$, we let $\fX_{\isom}$ denote the category with the same objects, but where the only morphisms are isomorphisms.

\begin{theorem} \label{thm:genmap}
The functor
\begin{equation} \label{eq:genmap}
\Fun^{\otimes}(\fP, \fT) \to \Fun(\bS(G)^{\op}, \Et(\fT))_{\isom}, \qquad \Phi \mapsto \Psi^{\circ}
\end{equation}
is fully faithful. Its essential image consists of functors $\Psi^{\circ}$ such that $\Psi$ is left-exact, additive, and compatible with $\mu$.
\end{theorem}

We note that any monoidal natural transformation between tensor functors with rigid source category is an isomorphism \cite[Proposition~1.13]{DeligneMilne}, which is why the target category in Theorem~\ref{thm:genmap} has the isom subscript.

We break the proof into two lemmas. For the first, we require some theory from \cite[\S 9]{repst}, which we now recall. A \defn{balanced functor} $\Omega \colon \bS(G) \to \fT$ is a pair of functors $\Omega_* \colon \bS(G) \to \fT$ and $\Omega^* \colon \bS(G)^{\op} \to \fT$ that have equal restriction to $\bS(G)_{\isom}$, which is canonically identified with its opposite. Suppose $\Omega$ is a balanced functor. For an object $X$ of $\bS(G)$, we write $\Omega(X)$ for the common value of $\Omega_*(X)$ and $\Omega^*(X)$. For a morphism $f \colon Y \to X$ in $\bS(G)$, we let $\alpha_f \colon \Omega(Y) \to \Omega(X)$ and $\beta_f \colon \Omega(X) \to \Omega(Y)$ be the morphisms provided by $\Omega_*$ and $\Omega^*$. In \cite[\S 9.2]{repst}, we introduced three important conditions on a balanced functor:
\begin{itemize}
\item We say that $\Omega$ is \defn{additive} if $\Omega(X \amalg Y)$ is identified with $\Omega(X) \oplus \Omega(Y)$ in the canonical manner, that is, if $i \colon X \to X \amalg Y$ and $j \colon Y \to X \amalg Y$ are the natural maps then $\alpha_i$, $\alpha_j$, $\beta_i$, and $\beta_j$ induce the direct sum decomposition.
\item We say that $\Omega$ satisfies \defn{base change} if whenever
\begin{displaymath}
\xymatrix{
Y' \ar[r]^{g'} \ar[d]_{f'} & Y \ar[d]^f \\
X' \ar[r]^g & X }
\end{displaymath}
is a cartesian square in $\bS(G)$, we have $\beta_g \alpha_f = \alpha_{f'} \beta_{g'}$.
\item We say that $\Omega$ is \defn{$\mu$-adapted} if whenever $f \colon Y \to X$ is a map of transitive $G$-sets we have $\alpha_f \beta_f=\mu(f) \cdot \id_{\Omega(Y)}$.
\end{itemize}
We have a natural balanced functor $\Omega_0 \colon \bS(G) \to \fP$, defined by $\Omega_0(X)=\cC(X)$, $(\Omega_0)_*(f)=f_*$, and $\Omega_0^*(f)=f^*$. This functor is additive, satisfies base change, and is $\mu$-adapted. If $\Omega$ is an arbitrary balanced functor satisfying these three properties then \cite[Proposition~9.3]{repst} states that there is a unique $k$-linear functor $\Phi \colon \fP \to \fT$ such that $\Omega=\Phi \circ \Omega_0$.

\begin{lemma}
Let $\Psi \colon \bS(G) \to \Et(\fT)^{\op}$ be a functor that is left-exact, additive, and compatible with $\mu$. Then $\Psi^{\circ}$ is in the essential image of \eqref{eq:genmap}.
\end{lemma}

\begin{proof}
We define a balanced functor $\Omega$ by putting $\Omega_*(-)=\Psi(-)^{\vee}$ and $\Omega^*(-)=\Psi(-)$. Note that since $\Psi(X)$ is an \'etale algebra it is a rigid object and canonically identified with its own dual. On objects, we have $\Omega(X)=\Psi(X)$. If $f \colon Y \to X$ is a map of $G$-sets then $\beta_f \colon \Psi(X) \to \Psi(Y)$ is the given algebra homomorphism $\Psi(f)$, and $\alpha_f \colon \Psi(Y) \to \Psi(X)$ is the dual map $\beta_f^{\vee}$, as in \S \ref{ss:dual}.

We now verify that $\Omega$ satisfies the three properties discussed above. Additivity follows directly from additivity of $\Psi$. Given a cartesian square as in the above discussion, we obtain a cartesian square
\begin{displaymath}
\xymatrix@C=3em{
\Psi(Y) \ar[r]^-{\Psi(g')} &  \Psi(Y')  \\
\Psi(X) \ar[r]^-{\Psi(g)} \ar[u]^{\Psi(f)} & \Psi(X') \ar[u]_{\Psi(f')} }
\end{displaymath}
of \'etale algebras since $\Psi$ is left-exact. Applying Proposition~\ref{prop:etale-bc}, we find that $\Omega$ satisfies base change. Now suppose $f \colon Y \to X$ is a map of transitive $G$-sets. Then $\Psi(Y)$ is a $\Psi(X)$-module via $\beta_f$, and $\alpha_f \colon \Psi(Y) \to \Psi(X)$ is a map of $\Psi(X)$-modules; in elemental notation, this means $\alpha_f(\beta_f(x) y)=x \alpha_f(y)$, for $x \in \Psi(X)$ and $y \in \Psi(Y)$. Applying this with $y=1$, we find $\alpha_f(\beta_f(x))=\alpha_f(1) \cdot x$ for $x \in \Psi(X)$. Since $\alpha_f(1)=\gamma(\Psi(f))=\mu(f)$, we see that $\Omega$ is $\mu$-adapted; note that here we have used the compatibility of $\Psi$ with $\mu$.

By \cite[Proposition~9.3]{repst}, we have a unique $k$-linear functor $\Phi \colon \fP \to \fT$ such that $\Omega=\Phi \circ \Omega_0$ as balanced functors. We claim that $\Phi$ is naturally a symmetric monoidal functor. We have natural isomorphisms
\begin{displaymath}
\Phi(\bbone) = \Phi(\cC(\bone)) = \Psi(\bone) = \bbone,
\end{displaymath}
where in the final step we used that $\Psi \colon \bS(G) \to \Et(\fT)^{\op}$ is left-exact, and thus preserves final objects. Let $X$ and $Y$ be $G$-sets. We have an isomorphism
\begin{displaymath}
i_{X,Y} \colon \Phi(\cC(X) \otimes \cC(Y)) \to \Phi(\cC(X)) \otimes \Phi(\cC(Y))
\end{displaymath}
by composing the isomorphisms.
\begin{displaymath}
\Phi(\cC(X) \otimes \cC(Y)) = \Phi(\cC(X \times Y)) = \Psi(X \times Y) = \Psi(X) \otimes \Psi(Y) = \Phi(\cC(X)) \otimes \Phi(\cC(Y)).
\end{displaymath}
In the third step above, we use that $\Psi$ is left-exact. Since the $i$ isomorphism is canonical, one easily sees that it is compatible with the associativity constraints. Since $i$ comes from the $\Psi$ functor, it is natural with respect to the $\beta$ maps. Since $i$ is an isomorphism, its dual coincides with its inverse (Proposition~\ref{prop:etale-inv}). Thus $i$ is also natural with respect to the $\alpha$ maps. Since the $\alpha$ and $\beta$ maps generate all maps (\S \ref{ss:uperm})), we see that $i$ is in fact a natural transformation. This shows that $\Phi$ has a symmetric monoidal structure.

The image of $\Phi$ under the functor \eqref{eq:genmap} is naturally identified with the functor $\Psi^{\circ}$. This completes the proof.
\end{proof}

\begin{lemma}
The functor \eqref{eq:genmap} is fully faithful.
\end{lemma}

\begin{proof}
Let $\Phi, \Phi' \colon \fP \to \fT$ be tensor functors, and let $\Psi^{\circ}$ and $(\Psi')^{\circ}$ be the functors coming via \eqref{eq:genmap}. Suppose that for each $G$-set $X$ we have an isomorphism
\begin{displaymath}
\alpha_X \colon \Phi(\cC(X)) \to \Phi'(\cC(X))
\end{displaymath}
in $\fT$. Consider the following conditions on the system $\alpha$:
\begin{enumerate}
\item $\alpha$ is natural with respect to pull-back maps, i.e., morphisms $f^* \colon \cC(X) \to \cC(Y)$ when $f \colon Y \to X$ is a map of $G$-sets.
\item $\alpha_X$ is an algebra isomorphism for each $X$.
\item $\alpha$ is natural with respect to all morphisms in $\fP$.
\item $\alpha$ is compatible with the monoidal structures.
\end{enumerate}
A monoidal natural isomorphism $\Phi \to \Phi'$ is a system $\alpha$ satisfying (c) and (d), while a natural isomorphism $\Psi^{\circ} \to (\Psi')^{\circ}$ is a system $\alpha$ satisfying (a) and (b). On morphisms, the functor \eqref{eq:genmap} simply takes the system $\alpha$ to itself\footnote{Since $\Phi$ is a well-defined functor it follows that (c) and (d) imply (a) and (b).}, and so it is faithful. To prove fullness, we must show that (a) and (b) imply (c) and (d).

Thus let $\alpha$ be a given system satisfying (a) and (b). We show that $\alpha$ satisfies (d). First, the monoidal unit of $\fP$ is $\cC(\bzero)$, and so as part of the data of a monoidal functor, we are given isomorphisms $\bbone \to \Phi(\cC(\bzero))$ and $\bbone \to \Phi'(\cC(\bzero))$. We must show that $\alpha_{\bzero}$ is compatible with these isomorphisms. However, this is clear since $\alpha_{\bzero}$ is an algebra homomorphism by (b), and these maps are the units for the algebra structures. Next, let $X$ and $Y$ be $G$-sets, let $Z=X \times Y$, and let $p_1 \colon Z \to X$ and $p_2 \colon Z \to Y$ be the two projections. Also let $m \colon \cC(Z) \otimes \cC(Z) \to \cC(Z)$ be the multiplication map. Consider the following diagram
\begin{displaymath}
\xymatrix@C=4em{
\Phi'(\cC(X)) \otimes \Phi'(\cC(Y)) \ar[r]^{p_1^* \otimes p_2^*} &
\Phi'(\cC(Z)) \otimes \Phi'(\cC(Z)) \ar[r]^-m &
\Phi'(\cC(Z)) \\
\Phi(\cC(X)) \otimes \Phi(\cC(Y)) \ar[r]^{p_1^* \otimes p_2^*} \ar[u]^{\alpha_X \otimes \alpha_Y} &
\Phi(\cC(Z)) \otimes \Phi(\cC(Z)) \ar[r]^-m \ar[u]^{\alpha_Z \otimes \alpha_Z} &
\Phi(\cC(Z)) \ar[u]^{\alpha_Z} }
\end{displaymath}
The left square commutes by (a) and the right square commutes by (b). The compositions in the two rows are the are isomorphisms in the monoidal structures for $\Phi$ and $\Phi'$, where here we identify $\cC(Z)$ with $\cC(X) \otimes \cC(Y)$. We have thus shown that (d) holds.

We now show that $\alpha$ satisfies (c). Since $\alpha_X$ is an algebra isomorphism, its dual is its inverse (Proposition~\ref{prop:etale-inv}). It follows that $\alpha$ is natural with respect to push-forwards. Since push-forwards and pull-backs generate all morphisms in $\fP$ (\S \ref{ss:uperm}), it follows that $\alpha$ is natural with respect to all morphisms in $\fP$.
\end{proof}

\begin{remark}
Theorem~\ref{thm:genmap} works just as well in the relative case; we briefly explain. Suppose $\sE$ is a stabilizer class for $G$ and $\mu$ is a measure for $G$ relative to $\sE$. Put $\fP=\uPerm(G, \sE, \mu)$. Then giving a tensor functor $\fP \to \fT$ is equivalent to giving a left-exact additive functor $\bS(G, \sE) \to \Et(\fT)^{\op}$ that is compatible with $\mu$ in the obvious sense. The proof is the same. Other results in \S \ref{s:genmap} apply in the relative case as well.
\end{remark}

\subsection{Fullness and faithfulness}

It is often important to understand when a tensor functor $\Phi \colon \fP \to \fT$ is full or faithful. We now give criteria for this in terms of the associated functor $\Psi$. We begin with a purely combinatorial result.

\begin{proposition} \label{prop:Psi-faithful}
Let $G$ be a pro-oligomorphic group, let $\cS$ be a lextensive category, and let $\Psi \colon \bS(G) \to \cS$ be an additive left-exact functor. The following are equivalent:
\begin{enumerate}
\item $\Psi$ is faithful.
\item $\Psi(X)=\bzero$ if and only if $X=\bzero$, for an object $X$ of $\bS(G)$.
\item For any object $X$ of $\bS(G)$, the map $\Psi \colon \Sub(X) \to \Sub(\Psi(X))$ is injective, where $\Sub(-)$ denotes the class of subobjects.
\end{enumerate}
\end{proposition}

\begin{proof}
(a) $\Rightarrow$ (b). Let $X$ be a non-empty object of $\bS(G)$. The switching map on $X \amalg X$ is then not the identity map. Since $\Psi$ is faithful, it follows that the switching map on $\Psi(X) \amalg \Psi(X)$ is not the identity. Thus $\Psi(X)$ is not empty.

(b) $\Rightarrow$ (c). First note that since $\Psi$ is left-exact, it preserves monomorphisms, and thus maps subobjects to subobjects. Let $A$ and $B$ be subobjects of $X$ such that $\Psi(A)=\Psi(B)$. We must show that $A=B$. First suppose that $A \subset B$. Then $B=A \amalg A'$, where $A'=B \setminus A$. We thus see that the map $\Psi(A) \to \Psi(A) \amalg \Psi(A') = \Psi(B)$ is an isomorphism. Since $\cS$ is extensive, it follows that $\Psi(A')=\bzero$. Thus, by (b), $A'=\bzero$, and so $A=B$, as required. To treat the general case, let $C=A \cap B$. We have
\begin{displaymath}
\Psi(C)=\Psi(A) \cap \Psi(B)=\Psi(A),
\end{displaymath}
where in the first step we use that $\Psi$ is left-exact, and in the second that $\Psi(A)=\Psi(B)$. Thus, by the previous case, we have $A=C$. The same argument shows $B=C$, and so $A=B$, as required.

(c) $\Rightarrow$ (a). Suppose $f,g \colon X \to Y$ are morphisms in $\bS(G)$ such that $\Psi(f)=\Psi(g)$. Let $\Gamma_f \subset Y \times X$ denote the graph of $f$, which we define as the equalizer of the maps $Y \times X \rightrightarrows Y$. Since $\Psi$ is left-exact, it commutes with formation of graphs. We thus have $\Psi(\Gamma_f)=\Psi(\Gamma_g)$, and so $\Gamma_f=\Gamma_g$ by (c). Thus $f=g$; to see this, note that $\Gamma_f$ as we have defined it agrees with the naive set-theoretic definition of the graph. This completes the proof.
\end{proof}

We now turn to tensor functors.

\begin{proposition} \label{prop:gen-faithful}
Let $\Phi \colon \fP \to \fT$ be a tensor functor, and let $\Psi \colon \bS(G) \to \Et(\fT)^{\op}$ be the associated functor. Then $\Phi$ is faithful if and only if $\Psi$ is.
\end{proposition}

\begin{proof}
Suppose $\Phi$ is faithful. If $X$ is a non-empty $G$-set then the identity map of $\cC(X)$ is non-zero, and so the identity map of $\Phi(\cC(X))=\Psi(X)$ is non-zero. Hence $\Psi(X)$ is not the empty object of $\Et(\fT)^{\op}$. Thus $\Psi$ is faithful (Proposition~\ref{prop:Psi-faithful}).

Now suppose that $\Psi$ is faithful. Let $X$ be a $G$-set, and consider the map
\begin{displaymath}
\Phi \colon \Gamma(\cC(X)) \to \Gamma(\Phi(\cC(X))).
\end{displaymath}
First suppose that $X$ is transitive. Then $\Gamma(\cC(X))$ is one dimensional, and spanned by the identity map $i$. Since the $\Psi(X)=\Phi(\cC(X))$ is non-zero (Proposition~\ref{prop:Psi-faithful}), it follows that $\Phi(i)$ is non-zero, and so $\Phi$ is injective. Now consider the general case. Write $X=X_1 \sqcup \cdots \sqcup X_n$, where each $X_i$ is transitive. Then $\Phi$ for $X$ is the direct sum of the corresponding maps on the $X_i$'s. Since each of these maps is injective, so is their sum. We have thus shown that for any object $A$ of $\fP$, the induced map
\begin{displaymath}
\Phi \colon \Hom_{\fP}(\bbone, A) \to \Hom_{\fT}(\bbone, \Phi(A))
\end{displaymath}
is injective. Since $\fP$ is rigid, it follows that $\Phi$ is faithful.
\end{proof}

\begin{proposition} \label{prop:gen-full}
Let $\Phi \colon \fP \to \fT$ be a tensor functor, and let $\Psi \colon \bS(G) \to \Et(\fT)^{\op}$ be the associated functor. Then $\Phi$ is full if and only if for every transitive $G$-set we have $\dim \Gamma(\Psi(X)) \le 1$.
\end{proposition}

\begin{proof}
Since $\fP$ is rigid, $\Phi$ is full if and only if the map
\begin{displaymath}
a_X \colon \Gamma(\cC(X)) \to \Gamma(\Phi(\cC(X)) = \Gamma(\Psi(X))
\end{displaymath}
is surjective for all objects $X$ of $\bS(G)$. By additivity, $a_X$ is surjective for all $X$ if and only if it is surjective for all transitive $X$. Let $X$ be a transitive $G$-set. Then $\Gamma(\cC(X))$ is the space of $G$-invariant $X \times \bone$ matrices, which is one dimensional since $G$ acts transitively on $X$. Thus if $a_X$ is surjective then $\Gamma(\Psi(X))$ is at most one dimensional. On the other hand, if $\Gamma(\Psi(X))$ is at most one dimensional then $a_X$ is surjective: indeed, since $\Phi$ is a tensor functor, $a_X$ is a $k$-algebra homomorphism, and therefore maps~1 to~1. This completes the proof.
\end{proof}

\subsection{A criterion for compatibility} \label{ss:crit}

Fix a faithful additive left-exact functor $\Psi \colon \bS(G) \to \Et(\fT)^{\op}$. The compatibility of $\Psi$ with a measure $\mu$ typically involves infinitely many conditions: we require $\gamma(\Psi(f))=\mu(f)$ for each map $f$ of transitive $G$-sets. We now show that it suffices to check this condition on a set of $\Theta$-generators, which can simplify the task enormously. In what follows, we let $\Sigma$ denote the class of morphisms $f \colon Y \to X$ in $\bS(G)$ with $X$ transitive, and we fix a set $S$ of $\Theta$-generators for $G$.

\begin{proposition} \label{prop:gamma-meas1}
Suppose $\Psi(f)$ is uniform for all $f \in \Sigma$. Then $f \mapsto \gamma(\Psi(f))$ defines a $k$-valued measure for $G$.
\end{proposition}

\begin{proof}
We must verify the three measure axioms from Definition~\ref{defn:meas}. Axiom (a) is clear, while (b) and (c) follow from \S \ref{ss:gamma}(d,e).
\end{proof}

\begin{proposition} \label{prop:gamma-meas2}
If $\Psi(f)$ is uniform for $f \in S$ then $\Psi(f)$ is uniform for all $f \in \Sigma$.
\end{proposition}

\begin{proof}
Let $\Pi \subset \Sigma$ be the set of $f$ such that $\Psi(f)$ is uniform. This is a $\Theta$-class by the results in \S \ref{ss:gamma}; precisely, axioms (b), (c), and (d) for $\Theta$-classes follow from \S \ref{ss:gamma}(d,e,f). Since $\Pi$ contains $S$ by assumption, we see that $\Pi=\Sigma$, which completes the proof.
\end{proof}

\begin{corollary} \label{cor:compatible}
Suppose that for all $f \in S$ the map $\Psi(f)$ is uniform with $\gamma(\Psi(f))=\mu(f)$. Then $\Psi$ is compatible with $\mu$.
\end{corollary}

\begin{proof}
Proposition~\ref{prop:gamma-meas2} implies that $\Psi(f)$ is uniform for all $f \in \Sigma$, and so Proposition~\ref{prop:gamma-meas1} implies that $f \mapsto \gamma(\Psi(f))$ is a $k$-valued measure for $G$. Since $\mu$ and $\gamma \circ \Psi$ are two $k$-valued measures that agree on $S$, they agree on all of $\Sigma$ by Corollary~\ref{cor:theta-gen}. Thus $\Psi$ is compatible with $\mu$.
\end{proof}

\begin{remark}
In the above discussion, we required $\Psi$ to be faithful. One way for $\Psi$ to be non-faithful (which seems to be typical) is that it could factor as
\begin{displaymath}
\xymatrix{
\bS(G) \ar[r]^{\Pi} & \bS(H) \ar[r]^-{\Phi'} \ar[r] & \Et(\fT)^{\op} }
\end{displaymath}
where $\Pi$ is a quotient and $\Phi'$ is faithful, additive, and left-exact. By ``quotient,'' we mean $\Pi$ is additive, left-exact, maps transitive sets to either transitive sets or $\bzero$, and hits every transitive set; essentially this means that the Fra\"iss\'e class for $H$ is a subclass for the one for $G$. Suppose $\Phi'$ sends maps of transitive $H$-sets to uniform maps. Then the above discussion shows that $\Phi'$ is compatible with a measure $\nu$ for $H$, and Theorem~\ref{thm:genmap} produces a tensor functor $\Phi' \colon \uPerm(H, \nu) \to \fT$. If $\nu$ measure extends to a measure $\mu$ on $\bS(G)$ then $\Psi$ will be compatible with $\mu$, and there will be a tensor functor $\Phi \colon \uPerm(G, \mu) \to \fT$ that factors through $\Phi'$. However, in general, $\nu$ need not extend to $\mu$.
\end{remark}

\subsection{Maps to oligomorphic tensor categories}

Suppose now that we have a second pro-oligomorphic group $H$ equipped with a $k$-valued measure $\nu$, and let $\fQ = \uPerm(H, \nu)$. We now examine what our mapping property for $\fP$ yields when the target category is $\fQ$. Let $f \colon Y \to X$ be a map of finitary $H$-sets, let  be the orbit decomposition of $X$, and let $f_i \colon Y_i \to X_i$ be the base change of $X$ to $X_i$. We say that $f$ is \defn{uniform} (with respect to $\nu$) if the map $f^* \colon \cC(X) \to \cC(Y)$ in $\fQ$ is uniform in the sense of \S \ref{ss:gamma}. Using \S \ref{ss:gamma}(b), we can describe this condition concretely as follows. Let $X=\bigsqcup_{i=1}^n X_i$ be the orbit decomposition of $X$, and let $f_i \colon Y_i \to X_i$ be the base change of $X$ to $X_i$. Then $f$ is uniform if and only if $n \ge 1$ and $\nu(f_i)$ is independent of $i$. In this case, we let $\nu(f)$ be the common value of $\nu(f_i)$.

\begin{proposition}
Let $\Psi \colon \bS(G) \to \bS(H)$ be an additive left-exact functor such that whenever $f \colon Y \to X$ is a map of transitive $G$-sets either $\Psi(X)=\bzero$, or the map $\Psi(f)$ is uniform (with respect to $\nu$) and $\nu(\Psi(f))=\mu(f)$. Then there is an associated tensor functor
\begin{displaymath}
\Phi \colon \fP \to \fQ, \qquad \cC(X) \mapsto \cC(\Psi(X)).
\end{displaymath}
If $\Psi'$ is a second such functor with associated tensor functor $\Phi'$ then we have a natural identification
\begin{displaymath}
\Isom(\Phi, \Phi') = \Isom(\Psi, \Psi'),
\end{displaymath}
where the left side is computed in $\Fun^{\otimes}(\fP, \fQ)$ and the right side in $\Fun(\bS(G), \bS(H))$.
\end{proposition}

\begin{proof}
The condition on $\Psi$ exactly means that $\Psi$ is compatible with $\mu$, and so the result follows from Theorem~\ref{thm:genmap}
\end{proof}

\begin{remark}
If the natural functor $\bS(H) \to \Et(\fQ)^{\op}$ is an equivalence then every tensor functor $\fP \to \fQ$ comes from the construction in the proposition. However, there are cases where $\bS(H) \to \Et(\fQ)^{\op}$ is not an equivalence (see \cite[Remark~7.3]{interp}), and then it is possible for there to be tensor functors $\fP \to \fQ$ that do not come from the proposition.
\end{remark}

\section{Universal properties for Deligne's category}

We now explain how to recover the well-known universal property of Deligne's interpolation category `$\underline{\Rep}S_t$' from our general Theorem~\ref{thm:genmap}. We only sketch the proofs here since the results are already known; the details are similar to those in \S \ref{s:delgp} and \S \ref{s:delmap}.

\subsection{The category of $\fS$-sets}

Let $\Omega$ be the set $\{1,2,\ldots\}$ of positive integers and let $\fS$ be the group of all permutations of $\Omega$. This action is easily seen to be oligomorphic. We introduce some notation (here $n \ge 0$ is an integer):
\begin{itemize}
\item We let $\Omega^{[n]}$ be the subset of $\Omega^n$ consisting of $n$-tuples with distinct coordinates; this is easily seen to be a transitive $\fS$-set.
\item We let $p_n \colon \Omega^{[n]} \to \Omega^{[n-1]}$ be the projection map omitting the final coordinate.
\item We let $\Omega^{(n)}$ be the set of $n$-element subsets of $\Omega$, which is isomorphic to $\Omega^{[n]}/\fS_n$, where the finite symmetric group $\fS_n$ acts by permuting coordinates.
\item We let $\sE=\sE(\Omega)$ be the stabilizer class defined by $\Omega$ (\S \ref{ss:relative}).
\end{itemize}
The following proposition records the relevant structural facts about $\fS$-sets.

\begin{proposition} \label{prop:sym-sets}
We have the following:
\begin{enumerate}
\item Any transitive $\sE$-smooth $\fS$-set is isomorphic to some $\Omega^{[n]}$.
\item Any morphism $\Omega^{[n]} \to \Omega^{[m]}$ is the projection onto some subset of coordinates; in particular $m \le n$.
\item Any morphism $\Omega^{[n]} \to \Omega^{[m]}$ factors into a sequence $f_{m+1} \circ \cdots \circ f_{n-1} \circ f_n$, where $f_i$ is isomorphic to $p_i$.
\item Any transitive $\fS$-set is isomorphic to $\Omega^{[n]}/\Gamma$ for some $n$ and some subgroup $\Gamma$ of $\fS_n$.
\end{enumerate}
\end{proposition}

\begin{proof}
The $\fS$-orbits on $\Omega^n$ are simply characterized by which coordinates are equal, and (a) follows from this. Statement (b) is easy to see directly, and (c) follows from (b). Statement (d) follows from the classification of open subgroups of $\fS$ given in \cite[Proposition~14.1]{repst}.
\end{proof}

\subsection{The mapping property for $\fS$}

Let $\cS$ be a lextensive category. We say that an object $X$ of $\cS$ is \defn{$\Delta$-complemented} if the diagonal $\Delta_X \to X \times X$ admits a complementary subobject. Suppose $X$ has this property. Write $X^{[2]}$ for the unique complement of $\Delta_X$. For $n \ge 3$, we define $X^{[n]} \subset X^n$ to be the intersection of $p^{-1}(X^{[2]})$ as $p \colon X^n \to X^2$ varies over all projection maps. We also put $X^{[1]}=X$ and $X^{[0]}=\bone$. It is not difficult to see that $X^n$ decomposes into a coproduct of objects that are isomorphic to $X^{[m]}$ for various $m$. See \S \ref{ss:ord-power} for a detailed proof of a related (and more complicated) claim. We say that $X$ is \defn{finite-like} if $X^{[n]}=\bzero$ for some $n$, and \defn{infinite-like} otherwise.

\begin{example}
Suppose $\cS$ is the category of sets. Then any object $X$ is $\Delta$-complemented. The object $X^{[n]}$ is the subset of $X^n$ where the coordinates are distinct. The object $X$ is finite-like if and only if $X$ is a finite set.
\end{example}

We now give a mapping property for the category $\bS(\fS, \sE)$. Let $\cS^{\Delta}$ be the full subcategory of $\cS$ spanned by the $\Delta$-complemented objects.

\begin{proposition} \label{prop:sym-map}
The functor
\begin{displaymath}
i \colon \LEx^{\oplus}(\bS(\fS, \sE), \cS) \to \cS^{\Delta}, \qquad \Psi \mapsto \Psi(\Omega)
\end{displaymath}
is an equivalence of categories. Moreover, a functor $\Psi$ is faithful if and only if the object $\Psi(\Omega)$ is infinite-like.
\end{proposition}

\begin{proof}
Suppose $X$ is a $\Delta$-complemented object. We define a functor $\Psi_X \colon \bS(\fS, \sE) \to \cS^{\Delta}$, as follows. We let $\Psi_X(\Omega^{[n]})=X^{[n]}$. We also define $\Psi_X(p_n)$ to be the obvious analog $X^{[n]} \to X^{[n-1]}$ of $p_n$. This determines $\Psi_X$ on the category of transitive objects, and we then extend to general objects by additivity. It is not difficult to verify that $\Psi_X$ is left-exact; a detailed proof in a similar case can be found in \S \ref{ss:delgp-univ}. We thus have a functor
\begin{displaymath}
j \colon \cS^{\Delta} \to \LEx^{\oplus}(\bS(\fS, \sE), \cS).
\end{displaymath}
It is not difficult to then verify that $i$ and $j$ are quasi-inverse; again, see \S \ref{ss:delgp-univ} for details in a related case. The statement about faithfulness follows from Propositions~\ref{prop:Psi-faithful} and~\ref{prop:gen-faithful}
\end{proof}

\begin{remark}
There does not seem to be a nice mapping property for additive left-exact functors out of the category $\bS(\fS)$. Indeed, suppose one has such a functor $\Psi \colon \bS(\fS) \to \cS$, and put $X=\Psi(\Omega)$. Recall that the transitive objects of $\bS(\fS)$ have the form $\Omega^{[n]}/\Gamma$, where $\Gamma$ is a subgroup of $\fS_n$. Since $\Psi$ is only left-exact, one cannot determine $\Psi(\Omega^{[n]}/\Gamma)$ from $X$ alone, except when $\Gamma$ is trivial. Thus one seems to need an infinite amount of data (with various relations) to describe such functors. There is a nice mapping property for additive \defn{exact} functors out of $\bS(\fS)$, though this is a very restrictive condition.
\end{remark}

\subsection{Measures}

Let $\bZ\langle x \rangle$ be the ring of integer-valued polynomials. This is the subring of $\bQ[x]$ generated (as a $\bZ$-module) by the binomial coefficients $\binom{x}{n}$. We have the following description of $\Theta$ rings.

\begin{proposition} \label{prop:sym-theta}
We have the following:
\begin{enumerate}
\item We have a ring isomorphism $\Theta(\fS) \cong \bZ\langle x \rangle$ under which $[\Omega^{(n)}]$ maps to $\binom{x}{n}$.
\item We have a ring isomorphism $\Theta(\fS, \sE) \cong \bZ[x]$ under which $[p_n]$ corresponds to $x-n+1$.
\end{enumerate}
\end{proposition}

\begin{proof}
(a) is \cite[Theorem~14.4]{repst}, and (b) follows from \cite[Proposition~14.15]{repst}.
\end{proof}

The proposition shows that for each $t \in k$ there is a unique $k$-valued measure $\mu_t$ for $(\fS, \sE)$ satisfying $\mu_t(\Omega)=t$. When $k$ has characteristic~0, the same is true in the absolute case (i.e., without the stabilizer class), but in positive characteristic the situation is more subtle. Proposition~\ref{prop:sym-theta}(b) shows that $[p_1]$ generates $\Theta(\fS, \sE)$, which suggests that $p_1$ could be a $\Theta$-generator for $\fS$ relative to $\sE$. We now verify that this is indeed the case.

\begin{proposition} \label{prop:sym-theta}
The map $p_1$ is a $\Theta$-generator for $\fS$ relative to $\sE$.
\end{proposition}

\begin{proof}
Let $\Sigma$ be the class of all maps $f \colon Y \to X$ in $\bS(\fS, \sE)$ with $X$ transitive. Let $\Pi \subset \Sigma$ be the $\Theta$-class generated by $p_1$. We must show that $\Pi=\Sigma$. It suffices, by axiom (d), to show that $\Pi$ contains all maps of transitive $\sE$-smooth $\fS$-sets. From the description of the category $\bS(\fS, \sE)$ in Proposition~\ref{prop:sym-sets}, and the axioms of $\Theta$-classes, it thus suffices to show that $\Pi$ contains $p_n$ for each $n$.

Suppose $\Pi$ contains $p_{n-1}$ for some $n \ge 1$. Consider the fiber product
\begin{displaymath}
\xymatrix@C=3em{
X \ar[r] \ar[d]_q & \Omega^{[n-1]} \ar[d]^{p_{n-1}} \\
\Omega^{[n-1]} \ar[r]^{p_{n-1}} & \Omega^{[n-2]}. }
\end{displaymath}
We can identify $X$ with the subset of $\Omega^n$ consisting of tuples $(x_1, \ldots, x_{n-2}, y, z)$ where all pairs of coordinates are distinct except perhaps $y$ and $z$. The map $q$ forgets the $z$ coordinate. We see that $X$ decomposes into the union of two orbits: $\Omega^{[n]}$, where $y$ and $z$ are distinct, and $\Omega^{[n-1]}$, where $y=z$. Moreover, $q$ is $p_n$ on $\Omega^{[n]}$ and the identity on $\Omega^{[n-1]}$. Since $p_{n-1}$ belongs to $\Pi$, so does $q$ by axiom (c). The identity map of $\Omega^{[n-1]}$ belongs to $\Pi$ by axiom (a). Thus $p_n$ belongs to $\Pi$ by axiom (d).

Since $\Pi$ contains $p_1$, the above argument inductively shows that $\Pi$ contains each $p_n$, and so the result follows.
\end{proof}

\subsection{The universal property} \label{ss:deligne-map}

The tensor category $\fP=\uPerm(\fS, \sE, \mu_t)$ is equivalent to Deligne's interpolation category $\underline{\Rep}S_t$ from \cite{Deligne}. For a tensor category $\fT$, let $\Et_t(\fT)$ denote the full subcategory of $\Et(\fT)$ spanned by algebras $A$ such that $\dim(A)=t$. We are now ready to give the universal property of $\fP$, as in \cite[Proposition~8.3]{Deligne}.

\begin{proposition}
The functor
\begin{displaymath}
i \colon \Fun^{\otimes}(\fP, \fT) \to \Et_t(\fT)_{\isom}, \qquad
\Phi \mapsto \Phi(\cC(\Omega))
\end{displaymath}
is an equivalence. Moreover, a tensor functor $\Phi$ is faithful if and only if the corresponding algebra $A=\Phi(\cC(\Omega))$ is infinite-like.
\end{proposition}

\begin{proof}
By Proposition~\ref{prop:sym-map}, giving an additive left-exact functor $\Psi \colon \bS(\fS, \sE) \to \Et(\fT)^{\op}$ amounts to giving an \'etale algebra $A$ in $\fT$; note that $\Et(\fT)^{\op}$ has complements (Proposition~\ref{prop:etale-cat}). The functor $\Psi$ is compatible with $\mu_t$ if and only if $\dim{A}=t$. If $A$ is infinite-like then this follows from Corollary~\ref{cor:compatible} and Proposition~\ref{prop:sym-theta}, while if $A$ is finite-like an additional argument is required. The fact that $i$ is an equivalence now follows from Theorem~\ref{thm:genmap}. The claim about faithfulness follows from Propositions~\ref{prop:Psi-faithful} and~\ref{prop:gen-faithful}.
\end{proof}

%\subsection{Knop's category}

%the following universal property for these categories: tensor functors from $\uRep(\GL_t(\bF_q))$ to a tensor category $\fT$ correspond to $\mathbb{F}_q$-linear Frobenius spaces in $\fT$ of dimension $t$ (\cite{EntovaAizenbudHeidersdorf}, Theorem 1.1.1). An $\mathbb{F}_q$-linear Frobenius space is an \'etale algebra with additional structure encoding the axioms of an $\mathbb{F}_q$-vector space, in much the same way that an ordered \'etale algebra encodes the axioms of a total order.  The full definition is somewhat involved, since the axioms for a vector space are more involved than those for a total order, so we not go into more detail here.

\begin{remark}
\begin{enumerate}
\item Knop defined categories $\uRep(\GL_t(\bF_q))$ interpolating the representation theory of finite general linear groups, see \cite{Knop1, Knop2}. These categories were further studied by Entova-Aizenbud and Heidersdorf, who proved a universal property in \cite{EntovaAizenbudHeidersdorf} in the sense that tensor functors from $\uRep(\GL_t(\bF_q))$ correspond to \'etale algebras with extra structure. Let $\bV=\bigcup_{n \ge 1} \bF_q^n$, let $G=\GL(\bV)$, and let $\sE=\sE(\bV)$ be the stabilizer class in $G$ defined by $\bV$. Then $\uRep(\GL_t(\bF_q))$ is the category $\uPerm(G, \sE, \mu_t)^{\rm kar}$ for an appropriate measure $\mu_t$. The universal property can, in principal, be recovered from Theorem~\ref{thm:genmap}, similar to the symmetric group case.
\item In \cite{interp}, tensor categories are attached to other infinite rank classical groups using the oligomorphic theory. One should be able to give universal properties for these categories using the approach suggested above.
\end{enumerate}

\end{remark}

\section{Universal property of the Delannoy group} \label{s:delgp}

In this section, we establish a universal property for the Delannoy group $\GG$, or, more precisely, the category $\bS(\GG)$. It states that additive left-exact functors $\bS(\GG) \to \cS$, with $\cS$ lextensive, correspond to (totally) ordered objects in $\cS$. Most of the work in this section is devoted to developing the theory of ordered objects in lextensive categories. In \S \ref{ss:ordered}--\S \ref{ss:order-functor} we give the definitions and establish the most fundamental properties. We define the category $\bS(\GG)$ in \S \ref{ss:delgp}, and then prove its universal property in \S \ref{ss:delgp-univ}. The remainder of the section provides some additional material on ordered objects.

\subsection{Ordered objects} \label{ss:ordered}

Fix, for the duration of \S \ref{s:delgp}, a lextensive category $\cS$ (\S \ref{ss:lextensive}). Let $X$ be an object of $\cS$. A \defn{binary relation} on $X$ is a subobject $R$ of $X \times X$. The \defn{opposite relation} $R^{\op}$ is the image of $R$ under the switching map $X \times X \to X \times X$. We say that $R$ is \defn{total} if the natural map
\begin{displaymath}
R \amalg \Delta_X \amalg R^{\op} \to X \times X
\end{displaymath}
is an isomorphism. We say that $R$ is \defn{transitive} if
\begin{displaymath}
R_{12} \cap R_{23} \subset R_{13},
\end{displaymath}
where $R_{ij}$ is the inverse image of $R$ under the projection $p_{ij} \colon X^3 \to X^2$. A \defn{total order} on $X$ is a binary relation that is total and transitive. A \defn{totally ordered object} of $\cS$ is a pair $(X, R_X)$, where $X$ is an object of $\cS$ and $R_X$ is a total order on $X$. We will typically drop the word ``total'' in what follows, and just speak of ``ordered objects.'' A morphism $f \colon X \to Y$ of ordered objects is \defn{monotonic} if $R_X \subset f^{-1}(R_Y)$. We let $\Ord(\cS)$ be the category of ordered objects and monotonic morphisms.

\begin{remark}
Intuitively, $R$ is the set of ordered pairs $(x,y)$ where $x<y$. The monotonic condition thus means that $f$ is strictly ordered preserving.
\end{remark}

\subsection{Functor of points} \label{ss:ord-func}

For many purposes, the definition of ordered object given above is somewhat cumbersome. A more flexible approach is provided through the functor of points. We now explain how this works.

Let $F$ be a pre-sheaf on $\cS$, i.e., a functor $\cS^{\op} \to \Set$. A \defn{(total) order} on $F$ consists of the data of a binary relation $<$ on $F(T)$ for each object $T$ of $\cS^{\op}$ such that the following conditions hold:
\begin{enumerate}
\item The relation $<$ on $F(T)$ is transitive for all $T$.
\item If $T\not= \bzero$ then the relation $<$ on $F(T)$ is anti-symmetric, i.e., at most one of $x<y$ and $x=y$ and $y<x$ is true.
\item For any morphism $f \colon T \to T'$, the function $f^* \colon F(T') \to F(T)$ is order-preserving, that is, $x<y$ implies $f^*(x)<f^*(y)$.
\item Given $a,b \in F(T)$, there is a decomposition $T=T_1 \sqcup T_2 \sqcup T_3$ such that $a \vert_{T_1}<b \vert_{T_1}$ and $a \vert_{T_2}=b \vert_{T_2}$ and $a \vert_{T_3}>b \vert_{T_3}$. Here $a \vert_{T_i}$ means the image of $a$ under the induced map $F(T) \to F(T_i)$.
\end{enumerate}
Here is a simple example, to provide some intuition. Take $\cS$ to be the category of sets, and $F(T)$ to be the space of real-valued functions on $T$. For $a,b \in F(T)$, we define $a<b$ if $a(x)<b(x)$ for all $x \in T$. In (d), $T_1$ consists of those $x \in T$ such that $a(x)<b(x)$. While we obtain a partial order on $F(T)$ by putting $a \le b$ if $a<b$ or $a=b$, this is somewhat misleading since $a \le b$ is not equivalent to $a(x) \le b(x)$ for all $x \in T$. This is why we prefer to work with strict orders in this setting.

We make one general observation. Fix a pre-sheaf $F$ with an order $<$.

\begin{proposition} \label{prop:order-d-unique}
The decomposition in (d) is unique.
\end{proposition}

\begin{proof}
Let $a,b \in F(T)$ be given, and suppose we have two decompositions
\begin{displaymath}
T = T_1 \sqcup T_2 \sqcup T_3 = T'_1 \sqcup T'_2 \sqcup T'_3
\end{displaymath}
as in (d). We then have
\begin{displaymath}
T = \coprod_{1 \le i,j \le 3} (T_i \cap T'_j).
\end{displaymath}
We have $a<b$ and $a>b$ on $T_1 \cap T'_3$, and so $T_1 \cap T'_3=\bzero$ by (b); more generally, $T_i \cap T'_j=\bzero$ whenever $i \ne j$. We thus find that the canonical inclusion
\begin{displaymath}
(T_1 \cap T_1') \sqcup (T_2 \cap T_2') \sqcup (T_3 \cap T_3') \to T_1 \sqcup T_2 \sqcup T_3
\end{displaymath}
is an isomorphism. In an extensive category, if $i \colon X \to X'$ and $j \colon Y \to Y'$ are morphisms such that $i \sqcup j$ is an isomorphism then $i$ and $j$ are each isomorphisms. We thus see that $T_i \cap T_i' \to T_i$ is an isomorphism for each $i$, and so $T_i \subset T_i'$. The reverse containment follows by symmetry, which completes the proof.
\end{proof}

For an object $X$ of $\cS$, let $h_X \colon \cS^{\op} \to \Set$ be the functor $\Hom_{\cS}(-, X)$. The following proposition is the main point of this discussion.

\begin{proposition}
There is a natural bijective correspondence between orders on $X$ and orders on $h_X$.
\end{proposition}

\begin{proof}
Suppose $R \subset X \times X$ is an order on $X$. Given $a,b \in h_X(T)$, we define $a<b$ if the product morphism $a \times b \colon T \to X \times X$ factors through $R$. We now check axioms for an order on $h_X$.
\begin{enumerate}
\item Suppose $a,b,c \in F(T)$ and $a<b$ and $b<c$. Consider the product morphism
\begin{displaymath}
a \times b \times c \colon T \to X \times X \times X.
\end{displaymath}
This factors through $R_{12}$ since $a<b$, and through $R_{23}$ since $b<c$, and therefore through $R_{12} \cap R_{23} \subset R_{13}$. This exactly means that $a<c$, and so $<$ is transitive.
\item Now suppose $T\not= \bzero$, and say $a,b \in F(T)$ satisfy $a<b$ and $a>b$. Then $a \times b \colon T \to X \times X$ maps into $R$ and $R^{\op}$, and therefore into $R \cap R^{\op}=\bzero$; this is a contradiction since $T\not=\bzero$ (initial objects are strict in extensive categories). The other cases (when $a=b$) are similar.
\item It is clear that pull-back morphisms are order-preserving.
\item Let $a,b \in F(T)$ be given, and consider the map $a \times b$ as above. Since $X \times X$ decomposes into $R \sqcup X \sqcup R^{\op}$, the requisite decomposition of $T$ follows from the fact that $\cS$ is extensive.
\end{enumerate}
We thus see that $<$ does indeed define an order on $h_X$.

Next, suppose we are given an order $<$ on $h_X$. Let $p,q \in h_X(X \times X)$ be the two projections. By axiom (d), we obtain a decomposition $X \times X = R \sqcup D \sqcup R'$, where $p \vert_R<q \vert_R$, $p \vert_D=q \vert_D$, and $p \vert_{R'}>q \vert_{R'}$. Clearly, $D$ contains the diagonal $\Delta_X$. We have
\begin{displaymath}
\Delta_X = (R \cap \Delta_X) \sqcup D \sqcup (R^{\op} \cap \Delta_X).
\end{displaymath}
We have $p=q$ and $p<q$ on $R \cap \Delta_X$, and so this intersection is empty; similarly for $R^{\op} \cap \Delta_X$. Thus $D=\Delta_X$. It is clear that $R'=R^{\op}$. We thus see that the relation $R$ is total. We now verify transitivity. Consider a morphism $(a,b,c) \colon T \to X^3$ that maps into $R_{12} \cap R_{23}$. We regard $a$, $b$, and $c$ as elements of $h_X(T)$. Since $(a,b,c)$ maps into $R_{12}$, it follows that $(a,b)$ maps into $R$, which exactly means $a<b$. Similarly $b<c$. Thus $a<c$ since the order on $h_X$ is transitive, which exactly means that $(a,b,c)$ means into $R_{13}$. We have thus shown that every $T$-point of $R_{12} \cap R_{23}$ factors through $R_{13}$, which means $R_{12} \cap R_{23} \subset R_{13}$, as required. We have therefore shown that $R$ is indeed a total order on $X$.

One easily sees that the two constructions are mutually inverse, which completes the proof.
\end{proof}

\subsection{Cartesian powers} \label{ss:ord-power}

Let $X$ be an ordered object of $\cS$. We can define various subobjects of $X^n$ by imposing relations between different coordinates. To systematically work with these subobjects, we introduce the following notion. An \defn{order scheme} is a finite set $S$ equipped with an equivalence relation $\sim$ and an order $<$ such that $<$ is compatible with the equivalence relation (i.e., equivalent elements have the same order), $<$ is transitive, and for any $x,y \in S$ at most one of $x<y$, $x>y$, or $x \sim y$ holds. If exactly one of these three possibilities hold, we say that the order scheme is \defn{maximal}. In other words, an order scheme is equivalent to the data of a preorder on a finite set, and it is maximal if the preorder it total.

Suppose $S$ is an order scheme. We define $X^S$ to be $X^n$ where $n=\# S$, but with the coordinates labeled by $S$. Given $x,y \in S$, we have the projection map $q_{x,y} \colon X^S \to X^2$ onto the $x$ and $y$ coordinates, with the $x$ coordinate is put first. We define $\Delta_{x,y}$ to be the inverse image of the diagonal, and $R_{x,y}$ to be the inverse image of $R$. We now define
\begin{displaymath}
X[S] = \big( \bigcap_{x \sim y} \Delta_{x,y} \big) \cap \big( \bigcap_{x<y} R_{x,y} \big).
\end{displaymath}
These are essentially all the natural subobjects of $X^S$ one can define using the order.

The most important instance of this construction comes by taking $S=[n]$ with order $1<2<\cdots<n$ and the trivial equivalence relation. In this case, we put $X^{(n)}=X[S]$. Thus, roughly speaking, $X^{(n)}$ is the subobject of $X^n$ where the coordinates are strictly increasing. If $\cS$ is the category of sets this is literally true: $X^{(n)}$ is the subset of $X^n$ consists of points $(x_1, \ldots, x_n)$ with $x_1<x_2<\cdots<x_n$.

If $S$ is any order scheme then there is an induced order scheme on $S/\hspace{-0.9mm}\sim$  that has trivial equivalence relation, and we have a natural isomorphism $X[S]=X[S/\hspace{-0.9mm}\sim]$. If $S$ is maximal then $S/\hspace{-0.9mm}\sim$ is a totally ordered set, and thus isomorphic to $[n]$ with the above order scheme where $n$ is the cardinality of $S/\hspace{-0.9mm}\sim$. We thus see that, if $S$ is maximal, then $X[S]$ is isomorphic to $X^{(n)}$ for some $n$.

Let $S$ be an order scheme. A \defn{refinement} of $S$ is an order scheme $S'$ on the same underlying set such that $x \sim_S y$ implies $x \sim_{S'} y$, and $x <_S y$ implies $x <_{S'} y$. Every order scheme has a maximal refinement, and often times many such refinements. If $S'$ is a refinement of $S$ then $X[S']$ is clearly a subobject of $X[S]$. Moreover, we have an isomorphism
\begin{displaymath}
\coprod_{S'} X[S'] \to X[S],
\end{displaymath}
where the coproduct is taken over the maximal refinements of $S$. This is easily seen using the functor of points perspective. In particular, we see that $X^n$ decomposes into a disjoint union of pieces, each of which are isomorphic to some $X^{(m)}$. (Note that $X^n=X[S]$ where $S=[n]$ has the trivial equivalence relation and trivial order, i.e., $x<y$ never holds.)

If $S$ and $S'$ are two order schemes, then the disjoint union $S \amalg S'$ carries a natural order scheme, where we introduce no relations between elements of $S$ and elements of $S'$. One easily sees that we have an isomorphism $X[S] \times X[S'] = X[S \amalg S']$. This is useful since it tells us how $X^{(n)} \times X^{(m)}$ decomposes: one considers the maximal refinements of the order scheme $[n] \amalg [m]$.

Let $S$ and $S'$ be order schemes. A \defn{strict injection} $j \colon S \to S'$ is an injective function such that the relations on $S$ are induced from those on $S'$, i.e., $x \sim y$ if and only if $j(x) \sim j(y)$, and $x<y$ if and only if $j(x)<j(y)$. Suppose we have such a map. The projection map $j^* \colon X^{S'} \to X^S$ then maps $X[S']$ into $X[S]$. In particular, if $j \colon [n] \to [m]$ is an order preserving map of finite sets, there is an induced map $X^{(m)} \to X^{(n)}$. For $1 \le i \le n$, we let $p_{n,i} \colon X^{(n)} \to X^{(n-1)}$ be the projection corresponding to the inclusion $j \colon [n-1] \to [n]$ such that $\im(j)$ does not contain $i$.

We rephrase part of the above discussion in the following manner. Let $\OI$ be the category of finite totally ordered sets with (strictly) order preserving maps. Then we have a functor $\OI^{\op} \to \cS$ that sends $[n]$ to $X^{(n)}$.

\begin{remark} \label{rmk:ord-finite}
An ordered object $X$ is necessarily $\Delta$-complemented, as $R \amalg R^{\op}$ provides a complement to $\Delta_X$ in $X \times X$. Using the functor of points perspective, it is not difficult to see that we have an isomorphism
\begin{displaymath}
X^{[n]} = \coprod_S X[S],
\end{displaymath}
where the coproduct is taken over maximal order schemes $S$ with trivial equivalence relations. Thus $X^{[n]}$ is isomorphic to $(X^{(n)})^{\amalg n!}$. In particular, we see that $X$ is finite-like if and only if $X^{(n)}=\bzero$ for some $n$.
\end{remark}

\subsection{Functors} \label{ss:order-functor}

We now examine how ordered objects behave under functors.

\begin{proposition} \label{prop:ord-func}
Let $\Psi \colon \cS' \to \cS$ be an additive left-exact functor of lextensive categories. Then $\Psi$ induces a functor
\begin{displaymath}
\Psi \colon \Ord(\cS') \to \Ord(\cS), \qquad (X, R_X) \mapsto (\Psi(X), \Psi(R_X)).
\end{displaymath}
Moreover, this functor is compatible with the order scheme constructions: we have a natural isomorphism $\Psi(X[S])=\Psi(X)[S]$, and if $i \colon S \to S'$ is a strict injection of order schemes then $\Psi(i^*)=i^*$.
\end{proposition}

\begin{proof}
This is a straightforward verification. The key point is that the definition of order and the order scheme constructions only refer to co-products and fiber products, and are thus compatible with additive left-exact functors.
\end{proof}

\subsection{The Delannoy group} \label{ss:delgp}

Let $\GG=\Aut(\bR, <)$ be the group of all order preserving self-bijections of the real line $\bR$. This group is oligomorphic via its action on $\bR$. We can view $\bR$ as an ordered object of the lextensive category $\bS(\GG)$. We therefore have an object $\bR^{(n)}$ for each $n \ge 0$. Explicitly, this is just the set of increasing tuples $(x_1, \ldots, x_n)$ in $\bR^n$, meaning $x_1<x_2<\cdots<x_n$. One easily sees that the action of $\GG$ on $\bR^{(n)}$ is transitive. In fact:

\begin{proposition}
The functor $\OI^{\op} \to \bS(\GG)$ given by $[n] \mapsto \bR^{(n)}$ is an equivalence onto the full subcategory of $\bS(\GG)$ spanned by transitive objects.
\end{proposition}

\begin{proof}
Every transitive $\GG$-set is isomorphic to $\bR^{(n)}$ for some $n$ by \cite[Corollary~16.2]{repst}. Let $x_n=(1, \ldots, n) \in \bR^{(n)}$, and let $H_n \subset \GG$ be the stabilizer of $x_n$. Then $\bR^{(n)} \cong \GG/H_n$. Thus giving a map $\bR^{(n)} \to \bR^{(m)}$ is equivalent to giving an $H_n$-fixed point on $\bR^{(m)}$. One easily sees that the fixed points are exactly those points $(y_1, \ldots, y_m)$ in $\bR^{(m)}$ with $\{y_1, \ldots, y_m\} \subset [n]$. This shows that every map $\bR^{(n)} \to \bR^{(m)}$ comes from a unique order preserving injection $[m] \to [n]$, which completes the proof.
\end{proof}

\subsection{The universal property} \label{ss:delgp-univ}

We now come to the main result of \S \ref{s:delgp}:

\begin{theorem} \label{thm:comb}
The functor
\begin{displaymath}
i \colon \LEx^{\oplus}(\bS(\GG), \cS) \to \Ord(\cS), \qquad \Psi \mapsto \Psi(\bR).
\end{displaymath}
is an equivalence of categories.
\end{theorem}

To prove the theorem, we define a functor in the opposite direction. Let $X$ be a totally ordered object of $\cS$. We then have a functor $\OI^{\op} \to \cS$ given by $[n] \mapsto X^{(n)}$. We also have a functor $\OI^{\op} \to \bS(\GG)$ given by $[n] \mapsto \bR^{(n)}$, which is an equivalence onto the category of transitive objects. It follows that there is a unique (up to isomorphism) additive functor $\Psi_X \colon \bS(G) \to \cS$ given on transitive objects by $\Psi_X(\bR^{(n)}) = X^{(n)}$.

\begin{lemma}
The functor $\Psi_X$ is left-exact.
\end{lemma}

\begin{proof}
Write $\Psi$ in place of $\Psi_X$ in what follows. For a maximal order scheme $S$, we let $i_S \colon \Psi(\bR[S]) \to X[S]$ be the natural isomorphism, obtained via the canonical identifications $\bR[S] \cong \bR^{(n)}$ and $X[S]=X^{(n)}$ where $n$ is the cardinality of $S/\hspace{-0.9mm}\sim$. Fix $x \in S$, and consider the following diagram
\begin{displaymath}
\xymatrix{
\Psi(\bR[S]) \ar[r] \ar[d]_{i_S} & \Psi(\bR^S) \ar[r] \ar[d]^j & \Psi(\bR) \ar@{=}[d] \\
X[S] \ar[r] & X^S \ar[r] & X }
\end{displaymath}
where $j$ is the canonical one and the right maps are the projections onto the $x$ coordinate. The right square commutes; indeed, this is essentially how $j$ is defined. The outer square also commutes, since the map $X[S] \to X$ comes from an $\OI$ map, specifically, the strict inclusion $\{x\} \to S$. It follows that the left square also commutes. Now consider the diagram
\begin{displaymath}
\xymatrix{
\Psi(\bR^n) \ar[d]_j \ar@{=}[r] & \coprod_S \Psi(\bR[S]) \ar[d] \\
X^n \ar@{=}[r] & \coprod_S X[S] }
\end{displaymath}
where $j$ is as before, $S$ varies over the maximal refinements of the trivial order scheme on $[n]$, and the right map is the coproduct of the $i_S$ isomorphisms. The diagram commutes by the above discussion. Thus $j$ is an isomorphism. This verifies that $\Psi$ is compatible with products in one particular case.

A slight modification of the above argument shows that for any order scheme $S$ we have a canonical isomorphism $\Psi(\bR[S]) \to X[S]$: simply identify $\bR[S]$ with a subobject of $\bR^S$, and use the decomposition of $\bR[S]$ into $\coprod \bR[S']$ with $S'$ varying over maximal refinements of $S$. If $S$ and $S'$ are two order schemes, then the diagram
\begin{displaymath}
\xymatrix{
\Psi(\bR[S] \times \bR[S']) \ar@{=}[r] \ar[d] & \Psi(\bR[S \amalg S']) \ar[d] \\
X[S] \times X[S'] \ar@{=}[r] & X[S \amalg S'] }
\end{displaymath}
is easily seen to commute. We thus see that the left map is an isomorphism. In particular, we see that for all $n,m \ge 0$ the natural map
\begin{displaymath}
\Psi(\bR^{(n)} \times \bR^{(m)}) \to \Psi(\bR^{(n)}) \times \Psi(\bR^{(m)})
\end{displaymath}
is an isomorphism. It thus follows that $\Psi$ is compatible with products. (Note that $\Psi$ preserves final objects, so it is indeed compatible with all finite products.)

Another minor modification gives compatibility with fiber products. Indeed, let $i \colon S \to S'$ and $j \colon S \to S''$ be strict injections of order schemes. Define $S' \amalg_S S''$ to be the minimal refinement of $S' \amalg S''$ in which $i(x) \sim j(x)$ for all $x \in S$. Then we have a natural identification
\begin{displaymath}
X[S'] \times_{X[S]} X[S''] \cong X(S' \amalg_S S'').
\end{displaymath}
The same argument used for products now gives compatibility of $\Psi$ with fiber products. The result thus follows.
\end{proof}

If $f \colon X \to Y$ is a monotonic map of ordered objects then $f$ induces maps $X^{(n)} \to Y^{(n)}$ for all $n$, and these maps are compatible with the $\OI$-morphisms. It follows that $f$ defines a natural transformation $\Psi_X \to \Psi_Y$. We therefore have a functor
\begin{displaymath}
j \colon \Ord(\cS) \to \LEx^{\oplus}(\bS(G), \cS), \qquad X \mapsto \Psi_X.
\end{displaymath}
We can now finish the proof of the theorem.

\begin{proof}[Proof of Theorem~\ref{thm:comb}]
It is clear that $i \circ j$ is the identity endofunctor of $\Ord(\cS)$. We now verify that the other composition is isomorphic to the identity. Let $\Psi \colon \bS(\GG) \to \cS$ be an additive left-exact functor, and let $X=\Psi(\bR)$. We have two functors $\OI^{\op} \to \cS$, namely, $[n] \mapsto X^{(n)}$ and $[n] \mapsto \Psi(\bR^{(n)}$. By Proposition~\ref{prop:ord-func}, they are isomorphic. The former functor can be written equivalently as $[n] \mapsto \Psi_X(\bR^{(n)})$. We thus see that $\Psi$ and $\Psi_X$ are isomorphic when restricted to the category of transitive objects in $\bS(\GG)$. Since both functors are additive, it follows that they are isomorphic. The isomorphism $\Psi \cong \Psi_X$ just obtained is clearly natural, and so we see that $j \circ i$ is isomorphic to the identity.
\end{proof}

We have the following companion result, which describes when functors are faithful.

\begin{proposition} \label{prop:comb-faithful}
Let $\Psi \colon \bS(\GG) \to \cS$ be a left-exact additive functor, and let $X=\Psi(\bR)$ be the associated ordered object of $\cS$. Then $\Psi$ is faithful if and only if $X$ is infinite-like, that is, $X^{(n)} \ne \bzero$ for all $n \ge 0$.
\end{proposition}

\begin{proof}
This follows directly from Proposition~\ref{prop:Psi-faithful}.
\end{proof}

\begin{example} \label{ex:comb}
Let $X=\bR \amalg \bone$ be the lexicographic sum of $\bR$ and $\bone$ in the category $\bS(\GG)$. Concretely, $X$ is simply $\bR$ with a maximal point $\infty$ added. By Theorem~\ref{thm:comb}, there is a left-exact additive functor $\Psi \colon \bS(\GG) \to \bS(\GG)$ satisfying $\Psi(\bR)=X$.

Recall that $p_{2,2} \colon X^{(2)} \to X$ is the map given by $(x,y) \mapsto x$. Since $\infty$ is the maximal point of $X$, there is no element $(\infty, y)$ in $X^{(2)}$. Thus the fiber of $p_{2,2}$ over $\infty$ is empty, and so $p_{2,2}$ is not surjective. Of course, the corresponding map $p_{2,2} \colon \bR^{(2)} \to \bR$ is surjective. We thus see that $\Psi$ does not preserve surjections. It follows that $\Psi$ is not induced from any group homomorphism $\GG \to \GG$, and also not exact.

This particular example is very relevant to Delannoy categories: we will see (\S \ref{ss:degone}) that it leads to a tensor functor $\fC_2 \to \fC_1$.
\end{example}

\begin{remark} \label{rmk:ord-univ}
Theorem~\ref{thm:comb} implies the existence of various kinds of universal formulas for ordered objects: essentially, any formula valid for $\bR$ in $\bS(\GG)$ will be valid for ordered objects in any lextensive category. For instance, one can directly verify that we have an isomorphism of $\bS(\GG)$-sets
\begin{displaymath}
(\bR^{(2)})^{(2)} \cong (\bR^{(4)})^{\amalg 3} \amalg (\bR^{(3)})^{\amalg 3}.
\end{displaymath}
It follows that for any ordered object $X$ in a lextensive category we have
\begin{displaymath}
(X^{(2)})^{(2)} \cong (X^{(4)})^{\amalg 3} \amalg (X^{(3)})^{\amalg 3}.
\end{displaymath}
Here we are using the lexicographic order on $X^{(2)}$; see \S \ref{ss:ord-constr}(e).
\end{remark}

\begin{remark}
Theorem~\ref{thm:comb} gives an equivalence
\begin{displaymath}
\LEx^{\oplus}(\bS(\GG), \bS(\GG)) = \Ord(\bS(\GG)).
\end{displaymath}
It follows that $\Ord(\bS(\GG))$ admits a natural monoidal structure, corresponding to composition on the left side. It also admits two other monoidal structures, coming from the lexicographic sum and product discussed in \S \ref{ss:ord-constr}. It would be interesting to investigate this category in more detail. For instance, can objects in this category be classified in any useful way?
\end{remark}

\subsection{Constructions of orders} \label{ss:ord-constr}

We now discuss various constructions of ordered objects.

\textit{(a) The reverse order.} If $(X, R_X)$ is an ordered object then so is $(X, R_X^{\op})$.

\textit{(b) The induced order.} Suppose that $Y$ is an ordered object and $X$ is a subobject of $Y$. Then $h_X(T) \subset h_Y(T)$ for all objects $T$. One easily verifies that endowing $h_X(T)$ with the induced $<$ relation from $h_Y(T)$ defines an order on $h_X$, and thus on $X$. We call this the \defn{induced order} on $X$.

\textit{(c) Subobjects of the final object.} Suppose that $X$ is a subobject of the final object. Then $h_X(T)$ is either empty or a singleton for all $T$. There is thus a unique anti-symmetric relation $<$ on $h_X(T)$. One readily verifies that this defines an order on $h_X$, and thus $X$. We thus see that $X$ admits a unique order.

\textit{(d) Lexicographic sum.} Let $X$ and $Y$ be ordered objects. We define an order on $X \amalg Y$, called the \defn{lexicographic sum}, by putting $X$ before $Y$. To be precise, suppose $a,b \in h_{X \amalg Y}(T)$ are given. Put
\begin{align*}
T_1 &= a^{-1}(X) \cap b^{-1}(X) & T_3 &= a^{-1}(Y) \cap b^{-1}(X) \\
T_2 &= a^{-1}(X) \cap b^{-1}(Y) & T_4 &= a^{-1}(Y) \cap b^{-1}(Y).
\end{align*}
Note that $T$ is the disjoint union of the $T_i$'s. We define $a<b$ if the following conditions hold:
\begin{itemize}
\item $a\vert_{T_1}<b \vert_{T_1}$ using the order on $h_X(T_1)$
\item $a\vert_{T_4}<b \vert_{T_4}$ using the order on $h_Y(T_4)$
\item $T_3=\bzero$.
\end{itemize}
We leave to the reader the routine verification that this does indeed define an order on $h_{X \amalg Y}$.

\textit{(e) Lexicographic product.} Let $X$ and $Y$ be ordered objects. We define an order on $X \times Y$, called the \defn{lexicographic product}, in the usual manner. To be precise, suppose $a,b \in h_{X \times Y}(T)$ are given. Write $a=(a_1,a_2)$ where $a_1 \in h_X(T)$ and $a_2 \in h_Y(T)$, and similarly write $b=(b_1,b_2)$. Let $T=T_1 \sqcup T_2 \sqcup T_3$ be the decomposition of $T$ such that $a_1<b_1$ on $T_1$, $a_1=b_1$ on $T_2$, and $a_1>b_1$ on $T_3$. We define $a<b$ if $T_3$ is empty and $a_2 \vert_{T_2}<b_2 \vert_{T_2}$. We leave to the reader the routine verification that this does indeed define an order on $h_{X \times Y}$.

\textit{(f) Ordered tuples.} Given an ordered object $X$, we have defined the subobject $X^{(n)}$ of $X^n$. It inherits the lexicographic order from $X^n$. For any permutation $\sigma \in \fS_n$, we can regard $X^{(n)}$ as a subobject of $X^n$ by composing the standard embedding with $\sigma$, which acts on $X^n$ by permuting coordinates, and then endow $X^{(n)}$ with the induced order. We call these the \defn{permlex} orders on $X^{(n)}$. The standard lexicographic order is the permlex order with $\sigma=1$. Another notably case is the reverse lexicographic order, which corresponds to the permutation $\sigma$ that reverses the elements of the set $[n]$, i.e., $\sigma(1)=n$, $\sigma(2)=n-2$, and so on. The reverse lexicographic order on $X^{(n)}$ is perhaps the most natural, since it compares the largest coordinates first.

\subsection{The decomposition associated to a point} \label{ss:pt-decomp}

Let $X$ be an ordered object in $\cS$ and suppose we have a morphism $a \colon \bone \to X$. This gives us a morphism
\begin{displaymath}
X = X \times \bone \stackrel{\id \times a}{\longrightarrow} X \times X = R \amalg \Delta_X \amalg R^{\op}.
\end{displaymath}
We thus obtain a decomposition of $X$
\begin{displaymath}
X = Y \sqcup \bone \sqcup Z,
\end{displaymath}
where $Y$ is the inverse image of $R$, $\bone$ is the inverse image of $\Delta_X$ (which maps to $X$ via $a$), and $Z$ is the inverse image of $R^{\op}$. Essentially by definition, $h_Y(T)$ consists of those elements $b \in h_X(T)$ such that $b<a$, and $h_Z$ is similarly described. From this, it follows that the above decomposition is a lexicographic sum, where $Y$ and $Z$ are equipped with the induced orders. We say that $a$ is \defn{maximal} if $Z=\bzero$ and \defn{minimal} if $Y=\bzero$.

\subsection{Finite-like orders}

Recall (Remark~\ref{rmk:ord-finite}) that an ordered object $X$ is finite-like if $X^{(n)}=\bzero$ for some $n$. The following result gives a nice characterization of these objects.

\begin{proposition} \label{prop:fin-ord}
Assume $\cS$ has complements, and let $X$ be an ordered object in $\cS$. The following are equivalent:
\begin{enumerate}
\item $X^{(n+1)}=\bzero$.
\item $X$ is isomorphic to a lexicographic co-product $X_1 \sqcup \cdots \sqcup X_n$, where $X_n \subset \cdots \subset X_1 \subset \bone$ are subobjects of the final object equipped with their unique orders.
\end{enumerate}
\end{proposition}

\begin{proof}
(b) $\Rightarrow$ (a). If $Y$ is a subobject of the final object then the diagonal $Y \to Y \times Y$ is an isomorphism, and so $Y^{(2)}=\bzero$. Now, $X^{(n+1)}$ decomposes into pieces of the form
\begin{displaymath}
X_1^{(a_1)} \times \cdots \times X_n^{(a_n)},
\end{displaymath}
where $a_1+\cdots+a_n=n+1$. Since some $a_i$ is at least~2, we have $X_i^{(a_i)}=\bzero$, and so every piece vanishes. Thus $X^{(n+1)}=\bzero$, as required.

(a) $\Rightarrow$ (b). First suppose $n=1$, i.e., $X^{(2)}=\bzero$. Recall that, by definition of order, the natural map
\begin{displaymath}
R \amalg \Delta_X \amalg R^{\op} \to X \times X
\end{displaymath}
is an isomorphism. Since $R=X^{(2)}$, we see that $R$ and $R^{\op}$ are empty. Thus the diagonal $X \to X \times X$ is an isomorphism. This means that every object has at most one map to $X$, and so the map $X \to \bone$ is a monomorphism, i.e., $X$ is a subobject of $\bone$. Thus (b) holds.

Now suppose $n \ge 2$. If $(a_1,\ldots,a_n)$ and $(b_1,\ldots,b_n)$ are distinct elements of $\bR^{(n)}$ then the set $\{a_i,b_i\}_{1 \le i \le n}$ has at least $n+1$ elements. Thus every $\GG$-orbit on $(\bR^{(n)})^{(2)}$ has the form $\bR^{(m)}$ with $m \ge n+1$. By the universality principle (Remark~\ref{rmk:ord-univ}), we see that for any ordered object $Y$ in a lextensive category we have a decomposition $(Y^{(n)})^{(2)} = \bigsqcup_{i=1}^N Y^{(m_i)}$ with each $m_i \ge n+1$. In particular, we see that $(X^{(n)})^{(2)}=\bzero$. Thus, by the $n=1$ case, $X^{(n)}$ is a subobject of $\bone$.

Write $E$ for $X^{(n)}$. By assumption, $E$ has a complement in $\bone$, that is, we have a decomposition $\bone=E \sqcup F$. By extensivity, we thus have $X=X_E \sqcup X_F$, where $X_E=X \times E$ and $X_F=X \times F$. Again, for the same reason, we have $E=X^{(n)}=X_E^{(n)} \sqcup X_F^{(n)}$, and so we see that $X_F^{(n)}=\bzero$. Thus, by induction, $X_F$ has the required form. We will return to this later; for now, we focus on $X_E$.

In what follows, we work in $\cS_{/E}$. Let $q_i \colon X^{(n)} \to X_E$ be the $i$th projection map. This is a monomorphism since $X^{(n)}=E$ is the final object. We will write $W_i$ for $q_i$, which we regard as a subobject of $X_E$; note that each $W_i$ is a copy of $E$, but the inclusions $W_i \to X_E$ are possibly different. We in fact claim that the $W_i$'s are disjoint. Let $q'_i \colon \bR^n \to \bR$ be the projection onto the $i$th coordinate. One easily sees that the fiber product of $q_i$ and $q_j$, for $i \ne j$, decomposes into $\GG$-orbits of the form $\bR^{(m)}$ with $m \ge n+1$. By the universality principle, the same statement holds for the $q_i$'s. Since $X_E^{(n+1)}=\bzero$, we see that $W_i \cap W_j=\bzero$ for $i \ne j$, as required.

Let $Y$ be the complement of the union of the $W_i$'s in $X_E$, so that we have a decomposition
\begin{displaymath}
X_E=W_1 \sqcup \cdots \sqcup W_n \sqcup Y.
\end{displaymath}
Now, $X^{(n+1)}=\bzero$ contains $W_1 \times \cdots \times W_n \times Y=E \times Y$ as a summand, and so $E \times Y=\bzero$. Since $E$ is the final object of $\cS_{/E}$, it follows that $Y=\bzero$. One easily sees that $W_n$ is a maximal point of $X_E$ (in the sense of \S \ref{ss:pt-decomp}), and so $X_E$ is the lexicographic sum of $W_1 \sqcup \cdots \sqcup W_{n-1}$ (with its induced order) and $W_n$ (with its unique order). Similarly, $W_{n-1}$ is a maximal point of $W_1 \sqcup \cdots \sqcup W_{n-1}$. Continuing in this way, we see that $X_E$ is the lexicographic sum of $W_1, \ldots, W_n$.

We now complete the proof. By induction, we have a lexicographic sum
\begin{displaymath}
X_F=W'_1 \sqcup \cdots \sqcup W'_{n-1},
\end{displaymath}
where each $W'_i$ is a subobject of $F$, and $W'_{i+1} \subset W'_i$. Put $X_i=W_i \sqcup W_i'$ for $1 \le i \le n-1$ and $X_n=W_n$. Then
\begin{displaymath}
X=X_1 \sqcup \cdots \sqcup X_n,
\end{displaymath}
is a lexicographic sum, and $X_n \subset \cdots \subset X_1$ are subobjects of $\bone$, as required.
\end{proof}

\begin{remark}
The order is necessary for Proposition~\ref{prop:fin-ord}; that is, if $X$ is a finite-like $\Delta$-complemented object then $X$ need not decompose into a co-product of subobjects of $\bone$. Indeed, suppose $G$ is a finite group and let $\cS=\bS(G)$ be the category of finite $G$-sets. This category has complements and every object is finite-like, however, not every object is a co-product of final objects; this is the case only for sets on which $G$ acts trivially.
\end{remark}

\section{Universal properties of the Delannoy categories} \label{s:delmap}

In this section, we establish the universal properties of the four Delannoy categories. These state that tensor functors $\fC_i \to \fT$ correspond to certain kinds of \'etale algebras in $\fT$ called Delannic algebras. We begin in \S \ref{ss:delcat} by defining the Delannoy categories. In \S \ref{ss:ordet} we define ordered \'etale algebras and prove some very basic results about them, and in \S \ref{ss:delannic} we do the same for Delannic algebras. The universal property is then proved in \S \ref{ss:delmap}. The remainder of the section provides some additional material on ordered and Delannic algebras.

\subsection{The Delannoy categories} \label{ss:delcat}

Recall that $\GG=\Aut(\bR, <)$ acts oligomorphically on $\bR$. We let $p_{n,i} \colon \bR^{(n)} \to \bR^{(n-1)}$ be the projection map omitting the $i$th coordinate. The ring $\Theta(\GG)$ that carries the universal measure for $\GG$ is isomorphic to $\bZ^4$. Thus $\GG$ admits exactly four $k$-valued measures $\mu_1, \ldots, \mu_4$. The following table gives their values on $p_{1,1}$, $p_{2,1}$ and $p_{2,2}$, which uniquely determines them.
\begin{center}
\begin{tabular}{l|rrrr}
& $\mu_1$ & $\mu_2$ & $\mu_3$ & $\mu_4$ \\
\hline
$p_{1,1}$ & $-1$ & 0 & 0 & \phantom{$-$}1 \\
$p_{2,1}$ & $-1$ & $-1$ &  $0$ & 0 \\
$p_{2,2}$ & $-1$ &  $0$ & $-1$ & 0
\end{tabular}
\end{center}
See \cite[\S 16]{repst} for proofs of the assertions made here.

We define the $i$th \defn{Delannoy catetory} to be
\begin{displaymath}
\fC_i = \uPerm(\GG, \mu_i)^{\rm kar},
\end{displaymath}
where the kar superscript denotes Karoubi envelope. The category $\fC_1$ is semi-simple pre-Tannakian, and was studied in depth in \cite{line}. The other three Delannoy categories are not abelian, and have not yet received much attention in the literature.

We now determine $\Theta$-generators for $\GG$.

\begin{proposition} \label{prop:del-theta-gen}
The maps $p_{1,1}$, $p_{2,1}$ and $p_{2,2}$ are $\Theta$-generators for $\GG$.
\end{proposition}

\begin{proof}
Let $\Sigma$ be as in \S \ref{ss:theta-gen}, and let $\Pi \subset \Sigma$ be the $\Theta$-class generated by $p_{1,1}$, $p_{2,1}$, and $p_{2,2}$. We must show that $\Pi=\Sigma$. It suffices, by axiom (d), to show that $\Pi$ contains all maps of transitive $\GG$-sets. Every map of transitive $\GG$-sets factors into a sequence of maps $p_{n,i}$, and so it suffices, by (b), to show that $\Pi$ contains the $p_{n,i}$. We are given that $\Pi$ contains all $p_{n,i}$ with $n \le 2$.

Consider the fiber product
\begin{displaymath}
\xymatrix{
X \ar[r] \ar[d]_q & \bR^{(2)} \ar[d]^{p_{2,2}} \\
\bR^{(2)} \ar[r]^{p_{2,1}} & \bR }
\end{displaymath}
The set $X$ is isomorphic to $\bR^{(3)}$ and the map $q$ is identified with $p_{3,3}$. Since $p_{2,2}$ belongs to $\Pi$, we see that $p_{3,3}$ belongs to $\Pi$ by axiom (c). Interchanging the roles of $p_{2,2}$ and $p_{2,1}$ above shows that $p_{3,1}$ belongs to $\Pi$.

Next, consider the similar fiber product
\begin{displaymath}
\xymatrix{
X \ar[r] \ar[d]_q & \bR^{(2)} \ar[d]^{p_{2,2}} \\
\bR^{(2)} \ar[r]^{p_{2,2}} & \bR }
\end{displaymath}
The set $X$ consists of all points $(x,y,z)$ in $\bR^3$ such that $x<y$ and $x<z$, and $q$ maps $(x,y,z)$ to $(x,y)$. There are three orbits on $X$: two are isomorphic to $\bR^{(3)}$, while the third is isomorphic to $\bR^{(2)}$. The restriction of $q$ to these orbits is $p_{3,2}$, $p_{3,3}$, and $\id$ respectively. We have already seen that $\Pi$ contains $p_{3,1}$. We also know that $\Pi$ contains $\id$, by axiom (a), and $q$ by axiom (c). Thus $\Pi$ contains $p_{3,2}$ by axiom (d).

We have now shown that $\Pi$ contains all $p_{n,i}$ with $n \le 3$. Every other $p_{n,i}$ can be obtained from one of these by an appropriate base change. Indeed, for $1<i<n$ we have a cartesian square
\begin{displaymath}
\vcenter{\vbox{\xymatrix{
\bR^{(n)} \ar[r] \ar[d]_{p_{n,i}} & \bR^{(3)} \ar[d]^{p_{3,2}} \\
\bR^{(n-1)} \ar[r]^f & \bR^{(2)} }}}
\qquad f(x)=(x_{i-1}, x_i).
\end{displaymath}
We also have a cartesian square
\begin{displaymath}
\vcenter{\vbox{\xymatrix{
\bR^{(n)} \ar[r] \ar[d]_{p_{n,1}} & \bR^{(2)} \ar[d]^{p_{2,1}} \\
\bR^{(n-1)} \ar[r]^f & \bR }}}
\qquad f(x)=x_1.
\end{displaymath}
There is a similar square for $p_{n,n}$. We thus see that $\Pi$ contains all $p_{n,i}$ by (c), which completes the proof.
\end{proof}

\begin{remark}
The two maps $p_{2,1}$ and $p_{2,2}$ are ``weak $\Theta$-generators'' as in Remark~\ref{rmk:weak-theta}. Note that this agrees with the fact that the measures are determined by their values on $p_{2,1}$, $p_{2,2}$.
\end{remark}

\subsection{Ordered \'etale algebras} \label{ss:ordet}

Fix, for the duration of \S \ref{s:delmap}, a Karoubian tensor category $\fT$. We now come to one of the key concepts of this paper:

\begin{definition}
An \defn{ordered \'etale algebra} in $\fT$ is an ordered object in the lextensive category $\Et(\fT)^{\op}$. We write $\OrdEt(\fT)$ for the category $\Ord(\Et(\fT)^{\op})^{\op}$ of ordered \'etale algebras in $\fT$.
\end{definition}

We make the definition more explicit. Let $A$ be an \'etale algebra. Recall that there is an idempotent $\sigma=\sigma_A$ in $\Gamma(A \otimes A)$ satisfying $(x \otimes 1) \sigma_A=(1 \otimes x) \sigma_A$ that provides a splitting of the multiplication map $A \otimes A \to A$. Giving an order on $A$ amounts to giving another idempotent $\tau=\tau_A$ in $\Gamma(A \otimes A)$ satisfying two conditions. First, we require an orthogonal decomposition
\begin{displaymath}
1 = \tau + \sigma + \tau^{\op},
\end{displaymath}
where $\tau^{\op}$ is the image of $\tau$ under the switching map on $\Gamma(A \otimes A)$. And second, we require $\tau_{1,3} \le \tau_{1,2} \tau_{2,3}$, where $\tau_{i,j}$ is the idempotent in $\Gamma(A \otimes A \otimes A)$ obtained by applying the map $A \otimes A \to A \otimes A \otimes A$ obtained from $\eta_A$ that maps the first $A$ to the $i$th factor and the second to the $j$th factor, and $e \le f$ means $ef=e$. If $(B, \tau_B)$ is a second ordered \'etale algebra, an algebra map $f \colon A \to B$ is monotonic if $f(\tau_A) \le \tau_B$. These are the morphisms in the category $\OrdEt(\fT)$.

The general constructions of ordered objects in \S \ref{ss:ord-constr} applies in particular to ordered \'etale algebras. Thus if $A$ and $B$ are ordered \'etale algebras then there is a lexicographic sum $A \oplus B$ and a lexicographic product $A \otimes B$. We also have the $A^{(n)}$ construction; note that $A^{(2)}$ is simply $\tau(A \otimes A)$. Recall that $A$ is said to be \defn{finite-like} if $A^{(n)}=0$ for some $n$; otherwise, we say $A$ is \defn{infinite-like}.

\begin{example}
The tensor unit $\bbone$ always has the structure of an ordered \'etale algebra. Taking lexicographic sums, we see that $\bbone^{\oplus n}$ has the structure of an ordered \'etale algebra. This construction defines a functor $\Ord(\FinSet) \to \OrdEt(\fT)^{\op}=\Ord(\Et(\fT)^{\op})$.
\end{example}

\begin{example} \label{ex:oea-alggp}
Let $\fT=\Rep(G)$ be the representation category of an algebraic group $G$. The category $\Et(\fT)^{\op}$ is equivalent to the category of finite $\pi_0(G)$-sets. It follows that an ordered \'etale algebra in $\fT$ corresponds to a finite $\pi_0(G)$-set equipped with a total order that is preserved by the group. Since $\pi_0(G)$ is a finite group, any such set must have trivial action. From this, it follows that the functor $\Ord(\FinSet) \to \OrdEt(\fT)^{\op}$ is an equivalence.
\end{example}

\begin{example}
The algebra $\cC(\bR)$ in the Delannoy category $\fC_i$ is an ordered \'etale algebra. Indeed, the functor
\begin{displaymath}
\bS(\GG) \to \Et(\fC_i)^{\op}, \qquad X \mapsto \cC(X)
\end{displaymath}
is additive and left-exact, and therefore carries ordered objects to ordered objects. The $\GG$-set $\bR$ is ordered, using the standard order on the real numbers. This is, as far as we know, the simplest example of a non-trivial ordered \'etale algebra.
\end{example}

The following proposition is the main reason we care about ordered \'etale algebras.

\begin{proposition} \label{prop:oea-univ-prop}
The functor
\begin{displaymath}
\LEx^{\oplus}(\bS(\GG), \Et(\fT)^{\op}) \to \OrdEt(\fT)^{\op}, \qquad
\Psi \mapsto \Psi(\bR)
\end{displaymath}
is an equivalence of categories. Moreover, a functor $\Psi$ is faithful if and only if the ordered \'etale algebra $\Psi(\bR)$ is infinite-like.
\end{proposition}

\begin{proof}
The first statement follows directly from Theorem~\ref{thm:comb}, while the second follows from Proposition~\ref{prop:comb-faithful}.
\end{proof}

We now characterize finite-like algebras, with the above proposition in mind.

\begin{proposition} \label{prop:finite-oea}
Let $A$ be an ordered \'etale algebra in $\fT$. The following are equivalent:
\begin{enumerate}
\item $A^{(n+1)}=0$.
\item $A$ is isomorphic to a lexicographic sum $A_1 \oplus \cdots \oplus A_n$ where $A_1$ is a direct factor of $\bbone$ and $A_i$ is a direct factor of $A_{i-1}$ for $2 \le i \le n$.
\end{enumerate}
\end{proposition}

\begin{proof}
This follows from Proposition~\ref{prop:fin-ord}. Note that $\Et(\fT)^{\op}$ has complements by Proposition~\ref{prop:etale-cat}.
\end{proof}

\subsection{Delannic algebras} \label{ss:delannic}

Let $A$ be an ordered algebra and let $\pi_A^i \colon A \to A^{(2)}$ for $i=1,2$ be the maps corresponding to the two projections; explicitly, $\pi_A^1(x)=\tau_A (x \otimes 1)$, and $\pi_A^2(x)=\tau_A (1 \otimes x)$. Put $\tilde{\gamma}_i(A)=\tilde{\gamma}(\pi_A^i)$, and drop the tilde when the map is uniform. The following is another important definition:

\begin{definition}
A non-zero ordered \'etale algebra $A$ is \defn{Delannic} if the unit $\eta_A$ and the two maps $\pi_A^1$ and $\pi_A^2$ are uniform (\S \ref{ss:gamma}). The zero algebra is also Delannic.
\end{definition}

The following proposition is trivial, but useful enough to record:

\begin{proposition} \label{prop:simple-delannic}
Let $A$ be an ordered \'etale algebra. If $\Gamma(\bbone)=\Gamma(A)=k$ then $A$ is Delannic.
\end{proposition}

It will sometimes be helpful to treat the finite-like case separately from the infinite-like case. The following proposition aids us in this.

\begin{proposition} \label{prop:finite-delannic}
A finite-like Delannic algebra is isomorphic to 0 or $\bbone$.
\end{proposition}

\begin{proof}
Let $A$ be a finite-like Delannic algebra. First suppose $A$ is a lexicographic sum $\bbone^{\oplus n}$. Then $\Gamma(A)=R^{\oplus n}$ where $R=\Gamma(\bbone)$. One easily sees that $\tilde{\gamma}_1(A)$ is the element $(0, 1, 2, \ldots, n-1)$ of $R^{\oplus n}$. Since this belongs to $k$, we have $n=0$ or $n=1$, as required.

We now treat the general case. By Proposition~\ref{prop:finite-oea}, we have a lexicographic sum $A=A_1 \oplus \cdots \oplus A_n$ for some $n$, where $A_1$ is a direct factor of $\bbone$ and $A_i$ is a direct factor of $A_{i-1}$ for $2 \le i \le n$. Write $\bbone=B \oplus C$ where $B=A_n$, so that $\fT$ decomposes as $\Mod_B \oplus \Mod_C$. Delannic algebras are clearly preserved under tensor functors, so the image of $A$ in $\Mod_B$ is Delannic. However, this is a sum of $n$ copies of the unit algebra $B$, and so $n \le 1$ by the previous case. We thus see that $A$ is a direct factor of $\bbone$. Since $\dim(A)$ belongs to $k$, we must have $A=0$ or $A=\bbone$.
\end{proof}

The next proposition is the reason Delannic algebras are important.

\begin{proposition} \label{prop:delannic-compatible}
Let $A$ be an ordered \'etale algebra and let $\Psi \colon \bS(\GG) \to \Et(\fT)^{\op}$ be the associated left-exact additive functor (Proposition~\ref{prop:oea-univ-prop}). Then $A$ is Delannic if and only if $\Psi$ is compatible with one of the measures $\mu_i$ for $1 \le i \le 4$. Moreover if $A \ne 0$ then $i$ is unique (if it exists).
\end{proposition}

\begin{proof}
First, observe that
\begin{displaymath}
\Psi(p_{1,1})=\eta_A, \qquad \Psi(p_{2,i})=\pi_A^i.
\end{displaymath}
Suppose that $\Psi$ is compatible with some measure. If $A \ne 0$ then, by definition, the above maps are uniform, and so $A$ is Delannic; of course, if $A=0$ then $A$ is Delannic too. Now suppose that $A$ is Delannic and infinite-like. Then $\Psi$ is faithful (Proposition~\ref{prop:oea-univ-prop}). Since $p_{1,1}$ and the $p_{2,1}$ are $\Theta$-generators for $\GG$ (Proposition~\ref{prop:del-theta-gen}) and sent to uniform maps, it follows that $\Psi$ is compatible with a unique measure (\S \ref{ss:crit}).

Finally, suppose $A$ is Delannic and finite-like. There are two cases: $A=0$ or $A=\bbone$ (Proposition~\ref{prop:finite-delannic}). In the first case, $A$ is compatible with $\mu_2$ and $\mu_3$ (and no other measures), while in the second case $A$ is compatible with $\mu_4$ (and no other measures).
\end{proof}

Let $A$ be a non-zero Delannic algebra. Then $A$ is compatible with exactly one of the $\mu_i$, and we say that $A$ has \defn{type $i$}. We have the following characterization of the various types (compare with the table of measures in \S \ref{ss:delcat}):
\begin{center}
\begin{tabular}{r|rrrr}
type & 1 & 2 & 3 & 4 \\
\hline
$\dim(A)$     & $-1$ &  $0$ &  $0$ & \phantom{$-$}$1$ \\
$\gamma_1(A)$ & $-1$ & $-1$ &  $0$ & $0$ \\
$\gamma_2(A)$ & $-1$ &  $0$ & $-1$ & $0$
\end{tabular}
\end{center}
Note in particular that if $A$ is Delannic then $\dim(A)$ is $\pm 1$ or 0, and that if $\dim(A) \ne 0$ then one can determine the type of $A$ solely from $\dim(A)$. The zero algebra is considered to have both type~2 and type~3, as it is compatible with $\mu_2$ and $\mu_3$. The unit algebra $\bbone$ is Delannic of type~4. Essentially by definition, the basic algebra $\cC(\bR)$ in $\fC_i$ is Delannic of type $i$. We let $\Del_i(\fT)$ be the full subcategory of $\OrdEt(\fT)$ spanned by Delannic algebras of type $i$.

\subsection{The universal property} \label{ss:delmap}

The following is our mapping property for Delannoy categories. It is one of the main results of this paper.

\begin{theorem} \label{thm:delmap}
The functor
\begin{displaymath}
\Fun^{\otimes}(\fC_i, \fT) \to \Del_i(\fT)_{\isom}, \qquad \Phi \mapsto \Phi(\cC(\bR))
\end{displaymath}
is an equivalence of categories.
\end{theorem}

\begin{proof}
This follows from the general mapping property for oligomorphic tensor categories (Theorem~\ref{thm:genmap}), combined with Propositions~\ref{prop:oea-univ-prop} and~\ref{prop:delannic-compatible}.
\end{proof}

We also have the following companion result.

\begin{proposition} \label{prop:delmap-ff}
Let $\Phi \colon \fC_i \to \fT$ be a tensor functor, and let $A=\Phi(\cC(\bR))$ be the associated Delannic algebra.
\begin{enumerate}
\item $\Phi$ is faithful if and only if $A$ is not 0 or $\bbone$.
\item$\Phi$ is full if and only if $\dim \Gamma(A^{(n)}) \le 1$ for all $n$.
\end{enumerate}
\end{proposition}

\begin{proof}
(a) $\Phi$ is faithful if and only if the associated functor $\Psi \colon \bS(\GG) \to \Et(\fT)^{\op}$ is faithful (Proposition~\ref{prop:gen-faithful}). The functor $\Psi$ is faithful if and only if $A$ is infinite-like (Proposition~\ref{prop:oea-univ-prop}). This, in turn, is equivalent to the stated condition (Proposition~\ref{prop:finite-delannic}).

(b) This follows from Proposition~\ref{prop:gen-full}.
\end{proof}

\begin{remark}
Theorem~\ref{thm:delmap} gives tensor functors
\begin{align*}
\fC_2 &\to \Vec & \fC_3 &\to \Vec & \fC_4 &\to \Vec \\
\cC_2(\bR) &\mapsto 0 & \cC_3(\bR) &\mapsto 0 & \cC_4(\bR) &\mapsto \bbone.
\end{align*}
These functors are manifestly full and essentially surjective, and therefore realize $\Vec$ as the semi-simplification of these $\fC_i$'s by \cite{BEEO} Lemma 2.6.
\end{remark}

\subsection{Operations on Delannic algebras} \label{ss:delop}

We now discuss some ways of constructing Delannic algebras. First, a simple observation: if $A$ is a Delannic algebra then so is $A$ equipped with the reverse order (\S \ref{ss:ord-constr}(a)). This preserves types~1 and~4, and interchanges types~2 and~3. This yields the following result.

\begin{proposition}
There is an equivalence of pre-Tannakian categories $\fC_2 \to \fC_3$.
\end{proposition}
\begin{proof}
Equipping $\cC_3(\bR)$ with its reverse order gives a type~2 algebra in $\fC_3$, and thus a tensor functor $\fC_2\to\fC_3$ via Theorem~\ref{thm:delmap}. By reversing the order  of $\cC_2(\bR)$ we get a tensor functor in the other direction and Theorem~\ref{thm:delmap} shows that they are mutually inverse.
\end{proof}

\begin{remark}
For $\fC_1$ and $\fC_4$ reversing the order instead induces a non-trivial auto-equi\-valence $\fC_i \to \fC_i$.  For $\fC_1$ this autoequivalence exchanges the simple objects $L_\wa$ and $L_\wb$; see \cite[Remark~4.17]{line}.
\end{remark}

Next we turn to lexicographic sum. Endow the set $\{1, 2, 3, 4\}$ with a partially defined binary operation $+$, as follows:
\begin{center}
\begin{tabular}{r|rrrr}
& 1 & 2 & 3 & 4 \\
\hline
1 &   &   & 1 & 2 \\
2 & 1 & 2 &   &   \\
3 &   &   & 3 & 4 \\
4 & 3 & 4 &   &
\end{tabular}
\end{center}
The first parameter corresponds to the row and the second to the column, e.g., $1+3=1$ but $3+1$ is undefined. The operation $+$ is associative (when defined), but not commutative. It has a natural geometric interpretation. Think of 1 as an open interval, 2 as a half-open interval with its right endpoint, 3 as a half-open interval with its left endpoint, and 4 as a closed interval. Then $i+j = \ell$ means that $\ell$ can be decomposed into a left piece of type $i$ and a right piece of type $j$.

\begin{proposition} \label{prop:delannic-sum}
Let $A$ and $B$ be non-zero ordered \'etale algebras, with lexicographic sum $A \oplus B$.
\begin{enumerate}
\item If $A$, $B$, and $A \oplus B$ are Delannic of types $i$, $j$, and $\ell$ then $i+j = \ell$.
\item If $A$ and $B$ are Delannic of types $i$ and $j$ and $i+j = \ell$ then $A \oplus B$ is Delannic of type $\ell$.
\item If $A \oplus B$ is Delannic and either $\dim(A)$ or $\dim(B)$ belongs to $k$ then both $A$ and $B$ are Delannic.
\end{enumerate}
\end{proposition}

Before giving the proof we require a lemma.

\begin{lemma} \label{lem:gamma-lex-sum}
Let $A$ and $B$ be ordered \'etale algebras, with lexicographic sum $A \oplus B$. Then
\begin{displaymath}
\tilde{\gamma}_1(A \oplus B)=(\tilde{\gamma}_1(A)+\eta_A(\dim(B)), \tilde{\gamma}_1(B)), \quad
\tilde{\gamma}_2(A \oplus B)=(\tilde{\gamma}_2(A), \eta_B(\dim(A))+\tilde{\gamma}_2(B))
\end{displaymath}
in $\Gamma(A) \oplus \Gamma(B)$. Here $\eta_A \colon \Gamma(\bbone) \to \Gamma(A)$ is the map induced by the unit $\eta_A \colon \bbone \to A$.
\end{lemma}

\begin{proof}
We first make a general observation. Let $\cS$ be a lextensive category, and let $X$ and $Y$ be ordered objects in $\cS$ with lexicographic sum $X \amalg Y$. Consider the first projection
\begin{displaymath}
X^{(2)} \amalg (X \times Y) \amalg Y^{(2)} \to (X \amalg Y)^{(2)} \to X \amalg Y.
\end{displaymath}
On $X^{(2)}$, this is the first projection onto $X$; on $X \times Y$, it is the projection onto $X$; and on $Y^{(2)}$ it is the first projection. Applying this with $\cS=\Et(\fT)^{\op}$ and $X=A$ and $Y=B$, we see that the map
\begin{displaymath}
\pi^1_{A \oplus B} \colon A \oplus B \to A^{(2)} \oplus (A \otimes B) \oplus B^{(2)}
\end{displaymath}
is given by
\begin{displaymath}
\pi^1_{A \oplus B}(a, b) = (\pi_A^1(a), i(a), \pi_B^1(b)),
\end{displaymath}
where $i \colon A \to A \otimes B$ is the natural map. It follows that
\begin{displaymath}
\tilde{\gamma}(\pi^1_{A \oplus B}) = (\tilde{\gamma}(\pi_A^1)+\tilde{\gamma}(i), \tilde{\gamma}(\pi^1_B)).
\end{displaymath}
Since $\tilde{\gamma}(i)=\eta_A(\dim(B))$ by base change \S \ref{ss:gamma}(d), the formula for $\tilde{\gamma}_1(A \oplus B)$ follows. The other formula is similar.
\end{proof}

\begin{proof}[Proof of Proposition~\ref{prop:delannic-sum}]
(b) We prove the case $1+4 = 2$; the others are similar. Thus suppose $A$ has type~1 and $B$ has type~4. By Lemma~\ref{lem:gamma-lex-sum}, we have
\begin{displaymath}
\tilde{\gamma}_1(A \oplus B) = (-1+1, 0)=0, \qquad
\tilde{\gamma}_2(A \oplus B) = (-1, -1+0) = -1,
\end{displaymath}
and so we see that the two maps $\pi^i_{A \oplus B}$ are uniform with $\gamma_1(A \oplus B)=0$ and $\gamma_2(A \oplus B)=-1$. Since 
\begin{displaymath}
\dim(A \oplus B) = \dim(A) + \dim(B) = (-1) + 1 = 0,
\end{displaymath}
we also have that $\eta_{A \oplus B}$ is uniform. The result follows.

(a) Since we have already proved (b), it is now enough to show that if $A$ and $B$ are Delannic of types $i$ and $j$ and no relation $i+j = \ell$ holds then $A \oplus B$ is not Delannic. This is again handled by considering the various cases. We treat the case $i=j=1$. By Lemma~\ref{lem:gamma-lex-sum}, we have
\begin{displaymath}
\tilde{\gamma}_1(A \oplus B) = (-2, -1),
\end{displaymath}
which does not belong to $k \subset \Gamma(A \oplus B)$. Thus $\pi^1_{A \oplus B}$ is not uniform, and so $A \oplus B$ is not Delannic.

(c) Observe that
\begin{displaymath}
\dim(A \oplus B) = \dim(A) + \dim(B)
\end{displaymath}
belongs to $k$ since $A \oplus B$ is Delannic, and so both $\dim(A)$ and $\dim(B)$ belong to $k$. From Lemma~\ref{lem:gamma-lex-sum}, we find
\begin{displaymath}
\tilde{\gamma}_1(A)=\tilde{\gamma}_1(A \oplus B) - \eta_A(\dim(B)), \quad
\tilde{\gamma}_2(A)=\tilde{\gamma}_2(A \oplus B),
\end{displaymath}
which both belong to $k$. Thus $\eta_A$, $\pi_A^1$, and $\pi_A^2$ are uniform, and so $A$ is Delannic. Similarly for $B$.
\end{proof}

We next examine the lexicographic product. Define a binary relation $\times$ on $\{1,2,3,4\}$ by:
\begin{center}
\begin{tabular}{r|rrrr}
& 1 & 2 & 3 & 4 \\
\hline
1 & 4 & 2 & 3 & 1 \\
2 & 3 & 2 & 3 & 2 \\
3 & 2 & 2 & 3 & 3 \\
4 & 1 & 2 & 3 & 4
\end{tabular}
\end{center}
Again, the first parameter corresponds to the row, e.g., $2 \times 3 = 3$ and $3 \times 2 = 2$. The operation $\times$ is associative but not commutative. The element~4 is the identity for $\times$, while the element~1 is an involution. The operation $\times$ distributes over the operation $+$, when the latter is defined. 
\begin{remark}
Intuition for this product comes from looking at $\bR^2$ with the lexicographic order, equipped with the product measure from the two factors. The set of points lexicographically larger than a point $(x,y)$ is a union of a half-line and a half-plane: $\{(a,y) \ | \ a >x \} \cup \{(c,d) \ | \ d > y\}$. Computing the measure of this gives either $-1$ or $0$ depending on the measures on the two factors. This, along with a similar calculation for set of points lexicographically smaller than $(x,y)$ determines the Delannic type of $\bR^2$. 
\end{remark}

\begin{proposition} \label{prop:delannic-prod}
If $A$ and $B$ are Delannic algebras of types $i$ and $j$ then the lexicographic product $A \otimes B$ is Delannic of type $i \times j$.
\end{proposition}

Again, we first require a lemma.

\begin{lemma}
Let $A$ and $B$ be ordered \'etale algebras with lexicographic product $A \otimes B$. Then
\begin{displaymath}
\tilde{\gamma}_1(A \otimes B) = \tilde{\gamma}_1(B)+\tilde{\gamma}_1(A) \dim(B), \quad
\tilde{\gamma}_2(A \otimes B) = \tilde{\gamma}_2(B)+\tilde{\gamma}_2(A) \dim(B)
\end{displaymath}
in $\Gamma(A \otimes B)$. Here we have implicitly mapped elements of $\Gamma(\bbone)$, $\Gamma(A)$, and $\Gamma(B)$ into $\Gamma(A \otimes B)$ in the canonical manner.
\end{lemma}

\begin{proof}
We first make a general observation. Let $\cS$ be a lextensive category, and let $X$ and $Y$ be ordered objects in $\cS$ with lexicographic product $X \times Y$. We have
\begin{displaymath}
(X \times Y)^{(2)} = (\Delta_X \times Y^{(2)}) \amalg (X^{(2)} \times Y^2).
\end{displaymath}
The first projection $(X \times Y)^{(2)} \to X \times Y$ is just the first projection from each component above, and similarly for the second component. Applying this with $\cS=\Et(\fT)^{\op}$ and $X=A$ and $Y=B$, we see that the map
\begin{displaymath}
\pi^1_{A \otimes B} \colon A \otimes B \to (A \otimes B^{(2)}) \oplus (A^{(2)} \otimes B \otimes B)
\end{displaymath}
is given by
\begin{displaymath}
\pi^1_{A \otimes B}(a,b) = (a \otimes \pi^1_B(b)) \oplus (\pi^1_A(a) \otimes b \otimes 1).
\end{displaymath}
The formula for $\tilde{\gamma}_1$ follows easily using properties of $\tilde{\gamma}$ from \S \ref{ss:gamma}. The case of $\tilde{\gamma}_2$ is similar.
\end{proof}

\begin{proof}[Proof of Proposition~\ref{prop:delannic-prod}]
First suppose that $\dim(B)=0$. Then
\begin{displaymath}
\dim(A \otimes B)=0, \quad \tilde{\gamma}_1(A \otimes B) = \tilde{\gamma}_1(B), \quad \tilde{\gamma}_2(A \otimes B) = \tilde{\gamma}_2(B).
\end{displaymath}
We thus see that if $B$ has type~2 (resp.~3) then $A \otimes B$ is Delannic of type~2 (resp.~3). This handles all case with $j=2$ or $j=3$.

Next, suppose that $B$ has type~4. Then
\begin{displaymath}
\dim(A \otimes B)=\dim(A), \quad \tilde{\gamma}_1(A \otimes B) = \tilde{\gamma}_1(A), \quad \tilde{\gamma}_2(A \otimes B) = \tilde{\gamma}_2(A).
\end{displaymath}
Thus $A \otimes B$ is Delannic of the same type as $A$. This handles all cases where $j=4$.

Finally, suppose that $B$ has type~1. Then
\begin{displaymath}
\dim(A \otimes B)=-\dim(A), \quad \tilde{\gamma}_1(A \otimes B) = -\tilde{\gamma}_1(A)-1, \quad \tilde{\gamma}_2(A \otimes B) = -\tilde{\gamma}_2(A)-1.
\end{displaymath}
Going case by case through the four options for $i$, we see that $A \otimes B$ has the stated type. This handles the all cases where $j=1$.
\end{proof}

\begin{remark}
In fact, the proof shows that if $B$ is Delannic of type~2 (resp.~3) and $A$ is any ordered \'etale algebra then $A \otimes B$ is Delannic of type~2 (resp.~3).
\end{remark}

We now examine the operation $A^{(n)}$ on ordered \'etale algebras, as discussed in \S \ref{ss:ord-constr}(f). For $n \ge 1$, define an operator $\lambda_n$ on the set $\{1,2,3,4\}$ as follows.
\begin{itemize}
\item $\lambda_n(1)$ is~1 if $n$ is odd and~4 if $n$ is even.
\item $\lambda_n(2)$ is~2 is $n$ is odd and~3 if $n$ is even.
\item $\lambda_n(3)$ is~3 for all $n$.
\item $\lambda_n(4)$ is~4 if $n=1$ and~3 if $n \ge 2$.
\end{itemize}
We also define $\lambda_0(n)=4$ for all $n$.

\begin{proposition} \label{prop:delannic-subset}
Let $A$ be a Delannic algebra and let $n \ge 0$.
\begin{enumerate}
\item If $A^{(n)}$ is endowed with any of the $n!$ permlex orders then $A^{(n)}$ is Delannic.
\item If $A$ has type~1 then $A^{(n)}$ has type $\lambda_n(1)$ under any permlex order.
\item If $A$ has type $i$ then $A^{(n)}$ has type $\lambda_n(i)$ under the lexicographic order.
\end{enumerate}
\end{proposition}

\begin{proof}
(a) Suppose that $A$ has type $i$. By Theorem~\ref{thm:delmap} we have a tensor functor $\Phi \colon \fC_i \to \fT$ mapping $\cC(\bR)$ to $A$. We thus see that $A^{(n)}=\Phi(\cC(\bR^{(n)}))$, where we use the same permlex order on $\cC(\bR^{(n)})$ as we use on $A^{(n)}$. Since $\Gamma(\bbone)=k$ and $\Gamma(\cC(\bR^{(n)}))=k$ hold in $\fC_i$, it follows that $\cC(\bR^{(n)})$ is Delannic (Proposition~\ref{prop:simple-delannic}), and thus so is $A^{(n)}$.

(b) The dimension of $\cC(\bR^{(n)})$ is the measure of the set $\bR^{(n)}$. If $i=1$, this is $(-1)^n$, and so the type is as described.

(c) This can be proved via explicit computations with $\cC(\bR)$.
\end{proof}

\begin{remark}
The set $\{1,2,3,4\}$ equipped with $+$ and $\times$ is a ring-like object. The $\lambda_n$ operations make it into something like a $\lambda$-ring.
\end{remark}

\subsection{The pre-Tannakian case}

We now study ordered \'etale algebras under the assumption that $\fT$ is pre-Tannakian. We note that part (b) in the following proposition can be generalised to arbitrary fields, by applying extension of scalars to the separable closure as in \cite{Del14}.

\begin{proposition} \label{prop:oea-dim}
Assume $k$ is separably closed. Let $A$ be an ordered \'etale algbera in $\fT$.
\begin{enumerate}
\item If $A$ is simple and, then $\dim(A)$ is $\pm 1$ or~0.
\item In general, $\dim(A)$ is an integer.
\end{enumerate}
\end{proposition}

\begin{proof}
(a) Since $\fT$ is pre-Tannakian we have $\Gamma(\bbone)=k$. Since $A$ is a simple \'etale algebra and $k$ is separably closed we have $\Gamma(A)=k$. Thus $A$ is Delannic (Proposition~\ref{prop:simple-delannic}), and so its dimension is as stated.

(b) Let $A=\bigoplus_{i=1}^n A_i$ be the decomposition of $A$ into simple \'etale algebras. Each $A_i$ inherits an ordered structure from $A$ (\S \ref{ss:ord-constr}(b)), and thus has dimension $\pm 1$ or~0 by part (a). Since $\dim(A)=\sum_{i=1}^n \dim(A_i)$, the result follows. 
\end{proof}

In fact, all three possibilities in (a) can occur. Indeed, the unit object in any tensor category provides an example with dimension $+1$, while the basic object $\cC(\bR)$ in the Delannoy category $\fC_1$ provides an example with dimension $-1$. The algebra $\cC(\bR^{(2)})$ in $\fC_1$ is also an example with dimension $+1$ that is not simply the tensor unit. It is more difficult to give an example of dimension~0, but we will do so in \S \ref{ss:dim0}.

\begin{proposition} \label{prop:oea-aut}
For an ordered \'etale algebra $A$ in $\fT$, the group $\Aut(A,<)$ is trivial.
\end{proposition}

\begin{proof}
The category $\Et(\fT)^{\op}$ is a pre-Galois category \cite{discrete}, i.e., equivalent to the category $\bS(G)$ for some pro-oligomorphic group $G$. Thus ordered \'etale algebras are anti-equivalent to ordered objects of $\bS(G)$. Since automorphism groups in $\bS(G)$ are finite \cite[Proposition~2.8]{repst} and a non-trivial finite group cannot preserve a total order, the result follows.
\end{proof}

\begin{corollary}
If $A$ and $B$ are isomorphic ordered \'etale algebras then there is a unique isomorphism $A \to B$.
\end{corollary}

\begin{remark}
Let $A$ be an ordered \'etale algebra. We can then consider the affine group scheme $G$ in $\operatorname{Ind}(\fT)$ defined by
\begin{displaymath}
G(T) = \Aut_T(T \otimes A, <),
\end{displaymath}
where here $T$ is an ind-algebra in $\fT$ and the right side is the automorphism group of the ordered \'etale algebra $T \otimes A$ in $\Mod_T$. Proposition~\ref{prop:oea-aut} shows that $G$ has no non-trivial points in (finite) \'etale algebras. On the other hand, $G$ is typically far from the trivial group scheme; for instance, if $\fT=\fC_1$ is the Delannoy category and $A=\cC(\bR)$ then $G$ is the fundamental group of $\fC_1$. In particular, we see that ordered \'etale algebras can have many automorphisms in categories like $\Mod_T$.
\end{remark}

\section{Examples and applications} \label{s:app}

We now give some examples of the universal property for $\fC_i$. In these examples, we will be constructing functors between different Delannoy categories. For clarity, we write $\cC_i(\bR^{(n)})$ for the Schwartz space on $\bR^{(n)}$ living in the category $\fC_i$.

\subsection{Simple examples} \label{ss:degone}

We begin with the simplest examples.

\begin{theorem}\label{thm:simplesfunctors}
There are tensor functors
\begin{align*}
\fC_2 &\to \fC_1, \qquad \cC_2(\bR) \mapsto \cC_1(\bR) \oplus \bbone, \\
\fC_3 &\to \fC_1, \qquad \cC_3(\bR) \mapsto \bbone \oplus \cC_1(\bR), \\
\fC_4 &\to \fC_1, \qquad \cC_4(\bR) \mapsto \bbone \oplus \cC_1(\bR) \oplus \bbone,\\
\fC_4 &\to \fC_2, \qquad \cC_4(\bR) \mapsto \bbone \oplus\cC_2(\bR),\\
\fC_4 &\to \fC_3, \qquad \cC_4(\bR) \mapsto  \cC_3(\bR)\oplus\bbone,
\end{align*}
where the algebras on the right are all lexicographic sums.
\end{theorem}

\begin{proof}
Indeed, the rightmost algebras are Delannic of types 2, 3, 4, 4 and 4 by Proposition~\ref{prop:delannic-sum}. Note that $\cC_1(\bR)$ is Delannic of type~1 and $\bbone$ is Delannic of type~4.
\end{proof}

Although these examples are very simple, they are significant since they show that the Delannoy categories $\fC_i$ with $2 \le i \le 4$ all fiber over the first Delannoy category $\fC_1$. This plays an important role in the analysis of $\fC_2$ in \cite{fake}.

Some interesting tensor functors of the form $\fC_i\to \fC_j\boxtimes\fC_l$ can be constructed by taking lexicographic sums of algebras isomorphic to $\cC_j(\bR)\boxtimes\bbone$, $\bbone$ and $\bbone\boxtimes\cC_l(\bR)$. We will view some of them in more detail below in \S \ref{ss:dim0}.

\subsection{A simple algebra of dimension~0} \label{ss:dim0}

We have seen that, assuming $k$ is separably closed, a simple ordered \'etale algebra in a pre-Tannakian category must have dimension $\pm 1$ or~0  (Proposition~\ref{prop:oea-dim}), and we exhibited such algebras of dimension $\pm 1$. We now do the same for dimension~0.

\begin{theorem} \label{thm:dim0}
For any field $k$, there is a pre-Tannakian category $\fF$ that contains a simple Delannic algebra $C$ of type~2.
\end{theorem}

In what follows, we let $A=\cC_1(\bR)$ and $B=\cC_2(\bR)$. We also let $A_1=A \boxtimes \bbone$ and $A_2=\bbone \boxtimes A$ in $\fC_1 \boxtimes \fC_1$.

\begin{lemma}
We have a commutative (up to isomorphism) square of tensor functors
\begin{displaymath}
\xymatrix{
\fC_2 \ar[r]^{\Phi_1} \ar[d]_{\Phi_2} & \fC_1 \ar[d]^{\Phi_3} \\
\fC_1 \boxtimes \fC_1 \ar[r]^{\Phi_4} & \fC_1 \boxtimes \fC_1 }
\end{displaymath}
where
\begin{align*}
\Phi_1(B) &= A \oplus A^{(2)} & \Phi_4(A_1) &= E \\
\Phi_2(B) &= A_1 \oplus A_2^{(2)} & \Phi_4(A_2) &= A_2 \\
\Phi_3(A) &= A_1 \oplus \bbone \oplus A_2
\end{align*}
and
\begin{displaymath}
E = A_1 \oplus \bbone \oplus A_2 \oplus A_1^{(2)} \oplus A_1 \oplus A_2 \oplus (A_1 \otimes A_2).
\end{displaymath}
The orders on these algebras are explained in the proof.
\end{lemma}

\begin{proof}
Throughout this proof, we use the lexicographic order on $(-)^{(2)}$. Suppose $X$ and $Y$ are totally ordered sets and $X \amalg Y$ is given the lexicographic sum order. We have
\begin{displaymath}
(X \amalg Y)^{(2)} = X^{(2)} \amalg (X \times Y) \amalg Y^{(2)}.
\end{displaymath}
The induced order on each individual summand is lexicographic. Moreover, the entire set is the lexicographic sum of $X^{(2)} \amalg (X \times Y)$ and $Y^{(2)}$, where each is given the induced order. However, $X^{(2)} \amalg (X \times Y)$ is not a lexicographic sum, in general. If we used the reverse lexicographic order the situation would be slightly different; for the present proof, this difference is very significant, and it is crucial that we use lexicographic order. The comments in this paragraph apply to ordered objects in any lextensive category.

Now, $A$ is a type~1 Delannic algebra and so $A^{(2)}$ is a type~4 Delannic algebra (Proposition~\ref{prop:delannic-subset}). Thus the lexicographic sum $A \oplus A^{(2)}$ has type~2 (Proposition~\ref{prop:delannic-sum}), and so the mapping property for $\fC_2$ provides the functor $\Phi_1$. The functors $\Phi_2$ and $\Phi_3$ are similar. 

Before discussing $\Phi_4$, we examine the composition $\Phi_3 \circ \Phi_1$. Since $\Phi_3$ is a tensor functor, it preserves natural operations on ordered algebras. Thus
\begin{displaymath}
\Phi_3(\Phi_1(B)) = \Phi_3(A) \oplus \Phi_3(A)^{(2)}
\end{displaymath}
is a lexicographic sum and $\Phi_3(A)^{(2)}$ carries the lexicographic order. Now, $\Phi_3(A)$ is the lexicographic sum of $(A_1 \oplus \bbone)$ and $A_2$. By the comments in the first paragraph, we thus have a lexicographic sum
\begin{displaymath}
\Phi_3(A)=E' \oplus A_2^{(2)}
\end{displaymath}
where
\begin{displaymath}
E'=(A_1 \oplus \bbone)^{(2)} \oplus (A_1 \oplus \bbone) \otimes A_2.
\end{displaymath}
We do not attempt to explicitly describe the order on $E'$, as it is unimportant. Since $A^{(2)}$ has type~4, so does $\Phi_3(A^{(2)})=E' \oplus A_2^{(2)}$. It follows from Proposition~\ref{prop:delannic-sum} that $E'$ is Delannic. Since $A_2^{(2)}$ and $E' \oplus A_2^{(2)}$ both have type~4, we see that $E'$ has type~3. Going back to the composition $\Phi_3 \circ \Phi_1$, we find
\begin{displaymath}
\Phi_3(\Phi_1(B)) = E \oplus A^{(2)},
\end{displaymath}
where
\begin{displaymath}
E=\Phi_3(A) \oplus E'
\end{displaymath}
is a lexicographic sum. Since $\Phi_3(A)$ has type~1 and $E'$ has type~3, it follows that $E$ has type~1 (Proposition~\ref{prop:delannic-sum}). Note that $E$ does decompose as in the statement of the lemma.

We now give the definition of $\Phi_4$. By the mapping property for $\fC_1$, we have tensor functors
\begin{displaymath}
\Phi_4', \Phi_4'' \colon \fC_1 \to \fC_1 \boxtimes \fC_1, \qquad \Phi'_4(A)=E, \quad \Phi''_4(A)=A_2.
\end{displaymath}
By the mapping property for the Deligne tensor product, there is therefore a tensor functor $\Phi_4$ with the stated properties.

We finally verify that the square commutes. We have
\begin{displaymath}
\Phi_4(\Phi_2(B))=\Phi_4(A_1) \oplus \Phi_4(A_2)^{(2)}=E \oplus A_2^{(2)},
\end{displaymath}
where the sum is lexicographic. We thus see that $\Phi_4(\Phi_2(B))$ and $\Phi_3(\Phi_1(B))$ are isomorphic ordered \'etale algebras. Therefore, by the mapping property for $\fC_2$, the tensor functors $\Phi_4 \circ \Phi_2$ and $\Phi_3 \circ \Phi_1$ are isomorphic.
\end{proof}

Consider the 2-fiber product
\begin{displaymath}
\xymatrix{
\fF \ar[r]^{\Pi_1} \ar[d]_{\Pi_2} & \fC_1 \ar[d]^{\Phi_3} \\
\fC_1 \boxtimes \fC_1 \ar[r]^{\Phi_4} & \fC_1 \boxtimes \fC_1 }
\end{displaymath}
Explicitly, an object of $\fF$ is a triple $(X, Y, i)$ where $X$ is an object of $\fC_1$, $Y$ is an object of $\fC_1 \boxtimes \fC_1$, and $i \colon \Phi_3(X) \to \Phi_4(Y)$ is an isomorphism. The category $\fF$ is pre-Tannakian; see Remark~\ref{rmk:pretan} below. The previous lemma furnishes us with a natural tensor functor
\begin{displaymath}
\Phi \colon \fC_2 \to \fF
\end{displaymath}
such that $\Pi_1 \circ \Phi \cong \Phi_1$ and $\Pi_2 \circ \Phi \cong \Phi_2$ as tensor functors. Let $C=\Phi(B)$. This is a type~2 Delannic algebra in $\fF$. The following lemma completes the proof of the theorem.

\begin{lemma}
The algebra $C$ is simple.
\end{lemma}

\begin{proof}
Suppose not. Since $\fF$ is pre-Tannakian, $C$ decomposes into a product of simple \'etale algebras. Since $\Pi_2$ is faithful, each of these simple factors remains a non-trivial factor of $\Pi_2(C)$. Since $\Pi_2(C)=A_1 \oplus A_2^{(2)}$ has two simple factors, it follows that $C$ must have exactly two simple factors, say $C_1$ and $C_2$. We label them so that $\Pi_2(C_1)=A_1$ and $\Pi_2(C_2)=A_2^{(2)}$.

Since $\Pi_1$ is also faithful, the same reasoning shows that $\Pi_1(C_1)$ and $\Pi_1(C_2)$ are non-trivial factors of $\Pi_1(C)$. Since $\Pi_1(C)=A \oplus A^{(2)}$ has only two simple factors, $\Pi_1(C_2)$ must be one of these simple factors. Thus we have $\Pi_1(C_2)=A$ or $\Pi_1(C_2)=A^{(2)}$.

Now, we have
\begin{displaymath}
\Phi_3(A) = A_1 \oplus \bbone \oplus A_2, \qquad
\Phi_3(A^{(2)}) = (A_1 \oplus \bbone \oplus A_2)^{(2)}.
\end{displaymath}
Neither of the above two algebras are isomorphic to $\Phi_4(A_2^{(2)})=A_2^{(2)}$, e.g. because the latter is a simple algebra. This is a contradiction, as we must have $\Phi_3(\Pi_1(C_2)) \cong \Phi_4(\Pi_2(C_2))$, by the definition of the 2-fiber product. We thus see that $C$ must be simple, which completes the proof.
\end{proof}

This is the first known example of a simple ordered \'etale algebra of dimension~0 in a pre-Tannakian category. We can in fact prove that $C^{(2)}$ is simple as well, and it is possible that $C^{(n)}$ is simple for all $n \ge 0$. In \cite{fake}, we construct an ordered \'etale algebra $D$ of dimension~0 in a pre-Tannakian category such that $D^{(n)}$ is simple for all $n \ge 0$. Theorem~\ref{thm:dim0} has an analog for $\fC_3$ requiring minimal changes. There is also an analog for $\fC_4$ that requires more substantive changes.
%This leads to a simple ordered \'etale algebra a pre-Tannakian category of dimension~1 that is not simply the tensor unit; again, this is the first known example of such an algebra.

\begin{remark} \label{rmk:pretan}
One can show directly that $\fF$ is pre-Tannakian, but here is perhaps a more compact argument using Deligne's theory of the fundamental group \cite[\S 8]{Deligne0}. Let $\pi_0$ be the fundamental group of $\fC_1 \boxtimes \fC_1$. Let $\pi$ and $\pi'$ be the automorphism group schemes of $\Phi_3$ and $\Phi_4$; there are natural maps $\epsilon \colon \pi_0 \to \pi$ and $\epsilon' \colon \pi_0 \to \pi'$. The category $\fC_1$ is equivalent to $\Rep(\pi, \epsilon)$, and moreover under this equivalence $\Phi_3$ corresponds to the forgetful functor $\Rep(\pi, \epsilon) \to \fC_1 \boxtimes \fC_1$. A similar comment applies to $\Phi_4$. We thus see that $\fF$ is equivalent to the category of triples $(X, a, a')$ where $X$ is an object in $\fC_1 \boxtimes \fC_1$, $a$ is an action of $\pi$ on $X$ that is compatible with $\pi_0$, and $a'$ is an action of $\pi'$ on $X$ that is compatible with $\pi_0$. This is clearly pre-Tannakian.
\end{remark}

\subsection{Abelian envelopes of Delannoy categories}

Consider the following two tensor functors:
\begin{displaymath}
\Phi_0, \Phi_1 \colon \fC_2 \to \fC_1, \qquad
\Phi_0(B) = A \oplus \bbone, \qquad \Phi_1(B) = A \oplus A^{(2)}.
\end{displaymath}
Since these are functors to a pre-Tannakian category, they factor through a local envelope.

\begin{theorem} \label{thm:two-envelopes}
The local envelopes for $\Phi_0$ and $\Phi_1$ are different.
\end{theorem}

\begin{proof}
Consider the map $p_{2,2}^* \colon B \to B^{(2)}$ in $\fC_2$. The map $\Phi_0(p_{2,2}^*)$ is not injective, while the map $\Phi_1(p_{2,2}^*)$ is injective. Since exact tensor functors between pre-Tannakian categories are also faithful (\cite{DeligneMilne} Proposition 1.19) it follows that $\Phi_0$ and $\Phi_1$ cannot both factor as composites of one tensor functor $\Phi:\fC_2\to \fT$ and exact tensor functors $\fT\rightrightarrows \fC_1$, as this would require $\Phi(p_{2,2}^\ast)$ both to be injective and not injective.
In particular the local envelopes must differ. 
\end{proof}

We therefore see that $\fC_2$ has at least two local envelopes. There are many other tensor functors from $\fC_2$ one can consider. For instance, for any $n \ge 0$ we have a tensor functor
\begin{displaymath}
\Phi_n \colon \fC_2 \to \fC_1, \qquad \Phi_2(B) = A \oplus A^{(2n)}.
\end{displaymath}
We have not been able to show that these lead to new envelopes. At the moment, it seems possible that all $\Phi_n$ with $n \ge 1$ have the same envelope, though we have little evidence for this.

\begin{remark}
In \cite{fake}, we will show that $\Phi_0$ is in fact already a local envelope. The functor $\Phi_1$ factors through the category $\fF$ above. This shows that the image of $B$ in the local envelope of $\Phi_1$ is a simple algebra, since its image in $\fF$ is simple. This gives an alternate proof that $\Phi_0$ and $\Phi_1$ have different local envelopes: in the first envelope, $B$ maps to a non-simple algebra, while in the second it maps to a simple algebra.
\end{remark}

\begin{remark}
Using similar reasoning to the above, one can show that $\fC_4$ has at least four local envelopes. Three of them relate to the local abelian envelopes of $\fC_2,\fC_3$ via the tensor functors $\fC_4\to\fC_2$ and $\fC_4\to\fC_3$ in Theorem~\ref{thm:simplesfunctors}.
\end{remark}

\end{document}